\documentclass[pdflatex,sn-mathphys-num]{sn-jnl}

\makeatletter
\def\email#1{\global\advance\emailcnt by 1\relax%
\if@corauemail%
   \g@addto@macro\corrauthemail{%
   \setcounter{footnote}{0}%
   \textcolor{blue}{#1}%
   }%
\else%
   \g@addto@macro\authemail{%
   \setcounter{footnote}{0}%
   \textcolor{blue}{#1}%
   }%
\fi}
\makeatother

\usepackage{graphicx}%
\usepackage{multirow}%
\usepackage{amsmath,amssymb,amsfonts}%
\usepackage{amsthm}%
\usepackage{mathrsfs}%
\usepackage[title]{appendix}%
\usepackage{xcolor}%
\usepackage{textcomp}%
\usepackage{manyfoot}%
\usepackage{booktabs}%
\usepackage{tikz}%
\usepackage{algorithm}%
\usepackage{algorithmicx}%
\usepackage{algpseudocode}%
\usepackage{listings}%

\theoremstyle{thmstyleone}%
\newtheorem{theorem}{Theorem}[section]%
\newtheorem{proposition}[theorem]{Proposition}%
\newtheorem{lemma}[theorem]{Lemma}%
\newtheorem{corollary}[theorem]{Corollary}%

\theoremstyle{thmstyletwo}%
\newtheorem{example}[theorem]{Example}%
\newtheorem{remark}[theorem]{Remark}%

\theoremstyle{thmstylethree}%
\newtheorem{definition}[theorem]{Definition}%

\theoremstyle{thmstyleone}%

\newtheorem{maintheorem}[theorem]{Theorem}

\usepackage{mathrsfs}
\usepackage{enumerate}
\usepackage{hyperref}

\newcommand{\calF}{\mathcal{F}}
\newcommand{\calH}{\mathcal{H}}

\newcommand{\calL}{\mathcal{L}}
\newcommand{\calM}{\mathcal{M}}
\newcommand{\calP}{\mathcal{P}}
\newcommand{\calU}{\mathcal{U}}

\newcommand{\R}{\mathbb{R}}
\newcommand{\N}{\mathbb{N}}

\newcommand{\calE}{\mathcal{E}}
\DeclareMathOperator{\dom}{dom}
\DeclareMathOperator{\diam}{diam}
\DeclareMathOperator{\Var}{Var}

\makeatletter
\renewcommand{\@maketitle}{%
    \vskip-0.1cm
    \hsize\textwidth\parindent0pt%
    {\hbox to \textwidth{{\Artcatfont\ArtType\hfill}\par}}%
    \ifx\@title\empty\else%
        \removelastskip\vskip20pt\nointerlineskip%
        {\Titlefont\@title\par}%
    \fi%
    \ifx\@subtitle\empty\else%
        \vskip9pt%
        {{\SubTitlefont\@subtitle\par}}%
    \fi%
    \ifnum\aucount>0
        \global\punctcount\aucount%
        \vskip20pt%
        \artauthors\par%
        {\vskip7pt\addressfont\auaddress\par%
	 \removelastskip\vskip24pt%
	\ifnum\emailcnt>0\relax%
           \ifx\corrauthemail\@empty\else{\ifnum\aucount>1*\fi}%
	   Corresponding author(s). E-mail(s): \corrauthemail\par\fi%
	   \ifx\authemail\@empty\else Contributing author:\ \authemail\fi%
        \fi%
        \ifequalcont{\par$^{\dagger}$\@equalconttext\par}\fi%
	 \removelastskip\vskip24pt%
        \ifpresentaddress{\par\@presentaddresstext\par}\fi%
	}
     \fi%
     {\printabstract\par}%
     {\printkeywords\par}%
     \ifx\@pacs\empty\else%
       \loop\ifnum\PacsCount>0%
          \csname\romannumeral\PacsTmpCnt StorePacsTxt\endcsname\par%
          \StepDownCounter{\PacsCount}%
          \StepUpCounter{\PacsTmpCnt}%
       \repeat%
    \fi%
    \removelastskip\vskip36pt\vskip0pt}%
\makeatother

\usepackage{tikz}%
\usetikzlibrary{decorations.pathreplacing}%
\usepackage{algorithm}%

\everymath{\displaystyle}

\begin{document}

\title{The Convex-Analytic Structure of Thermodynamic Equilibrium: Pressure, Subdifferentials, and Phase Transitions}


\author{\fnm{Abdoulaye} \sur{Thiam}}\email{athiam@allenuniversity.edu}


\affil{\orgdiv{Division of Mathematics and Natural Sciences}, \orgname{Allen University}, \orgaddress{\street{1530 Harden Street}, \city{Columbia}, \postcode{29204}, \state{South Carolina}, \country{USA}}}

\abstract{We develop the convex-analytic structure of the thermodynamic formalism for continuous maps on compact metric spaces. The pressure functional is the Legendre-Fenchel transform of the negative entropy, and the biconjugate recovery of the entropy from the pressure establishes a complete duality. Equilibrium states are elements of the subdifferential of the pressure, uniqueness of equilibrium states corresponds to G\^{a}teaux differentiability, and first-order phase transitions correspond to non-differentiability. For systems with specification and H\"{o}lder potentials, the pressure is Fr\'{e}chet differentiable in the H\"{o}lder norm, and the second derivative of the pressure equals the asymptotic variance of the Birkhoff sums. We prove a universal variational principle that unifies the classical additive, the subadditive, and the relative variational principles through a single theorem on convex functionals satisfying convexity, lower semi-continuity, coercivity, and cocycle invariance. Extensions to systems with the specification property and to non-compact spaces under coercivity conditions are included, with applications to countable Markov shifts via Sarig's recurrence classification. This Part constitutes Part~II of a six-part series on the thermodynamic formalism for hyperbolic dynamical systems.}



\maketitle

\begin{center}
\textit{Dedicated to the memory of Jean-Christophe Yoccoz (1957--2016),}\\
\textit{Fields Medalist and Professor at the Coll\`{e}ge de France, with whom the author}\\
\textit{had the privilege of working, and who introduced him to hyperbolic dynamics.}
\end{center}

\thispagestyle{empty}

\makeatletter
\renewcommand{\ps@headings}{%
  \def\@oddfoot{\hfill\thepage\hfill}%
  \let\@evenfoot\@oddfoot%
  \def\@evenhead{\hfill\normalfont\small\textit{A.~Thiam}\hfill}%
  \def\@oddhead{\hfill\normalfont\small\textit{The Convex-Analytic Structure of Thermodynamic Equilibrium}\hfill}%
  \let\@mkboth\markboth%
}
\pagestyle{headings}
\makeatother

\setlength{\parskip}{0.1em}

\section{Introduction}\label{sec:introduction}

This Part develops the convex-analytic structure underlying the thermodynamic formalism for continuous maps on compact metric spaces. The pressure functional $P: C(X) \to \R$ satisfies
\begin{equation}\label{eq:pressure_legendre}
P(\phi) = \sup_{\mu \in \calM_T(X)} \left\{ \int_X \phi \, d\mu + h_\mu(T) \right\},
\end{equation}
where $\calM_T(X)$ denotes the space of $T$-invariant Borel probability measures and $h_\mu(T)$ is the measure-theoretic entropy (the topological entropy in the supremum was introduced independently by Bowen \cite{Bowen1971} and Dinaburg \cite{Dinaburg1970}). Defining $S(\mu) = -h_\mu(T)$, equation \eqref{eq:pressure_legendre} becomes $P(\phi) = S^*(\phi)$: the pressure is the Legendre-Fenchel transform of the negative entropy. From this identification, equilibrium states are elements of the subdifferential $\partial P(\phi)$, uniqueness of the equilibrium state corresponds to differentiability of $P$ at $\phi$, and phase transitions correspond to non-differentiability.

The identification of pressure as a convex conjugate of entropy is due to Israel \cite{Israel1979} in the lattice gas context and to Ruelle \cite{Ruelle2004} in the dynamical setting. The thermodynamic formalism itself originates with Ruelle \cite{Ruelle1968, Ruelle1973} and Sinai \cite{Sinai1972}, who connected the statistical mechanics of lattice systems to the ergodic theory of hyperbolic dynamics. The subdifferential characterization of equilibrium states appears implicitly in Walters \cite{Walters1982} and Keller \cite{Keller1998}. The equivalence between non-differentiability and non-uniqueness is in Ruelle \cite{Ruelle2004} and Israel \cite{Israel1979}. Our contribution is twofold. First, we develop the full subdifferential theory (Main Theorem~\ref{thm:subdifferential_main}) as a self-contained treatment for general topological dynamical systems on compact spaces, including the biconjugate recovery $S = P^*$, the exposed point analysis, and explicit directional derivative formulas, which has not been assembled in this form in the dynamical systems literature. Second, we prove a Universal Variational Principle (Main Theorem~\ref{thm:universal_VP}) that unifies the classical variational principle of Walters \cite{Walters1975}, the subadditive variational principle of Cao et~al. \cite{CaoFengHuang2008} (see also Barreira \cite{Barreira1996} for the non-additive formalism and Feng-Huang \cite{FengHuang2012} for the Lyapunov spectrum of sub-additive potentials), and the relative variational principle of Ledrappier and Walters \cite{LedrappierWalters1977} through a single abstract theorem on convex functionals.

The contributions of this Part are fivefold, corresponding to its five Main Theorems. First, we prove that the pressure functional $P: C(X) \to \R$ admits the identification $P = S^*$ with the Legendre-Fenchel transform of the negative entropy $S(\mu) = -h_\mu(T)$ (Main Theorem~\ref{thm:pressure_legendre_main}), together with the full list of convex-analytic properties: convexity, Lipschitz continuity with constant~$1$, monotonicity, normalization $P(\phi+c) = P(\phi) + c$, and cohomological invariance $P(\phi\circ T) = P(\phi)$. Second, we characterize equilibrium states as elements of the subdifferential $\partial P(\phi)$ (Main Theorem~\ref{thm:subdifferential_main}), establishing four equivalent conditions: the equilibrium identity $h_\mu(T)+\int\phi\,d\mu = P(\phi)$, the subdifferential inclusion $\mu \in \partial P(\phi)$, the reverse inclusion $\phi \in \partial S(\mu)$, and the Fenchel-Young equality $P(\phi)+S(\mu) = \langle\phi,\mu\rangle$. Third, we prove that G\^{a}teaux differentiability of the pressure at $\phi$ is equivalent to uniqueness of the equilibrium state (Main Theorem~\ref{thm:differentiability_main}), with the derivative given by $DP(\phi)(\psi) = \int\psi\,d\mu_\phi$; for systems with specification and H\"{o}lder potentials, Fr\'{e}chet differentiability in the $\|\cdot\|_\alpha$ norm holds, and the second derivative $P''(\phi;\psi) = \sigma^2$ equals the asymptotic variance of the Birkhoff sums, connecting convex analysis directly to the central limit theorem. Fourth, we characterize first-order phase transitions through non-differentiability of the pressure (Main Theorem~\ref{thm:phase_transition_main}), establishing the equivalence of non-differentiability, non-singleton subdifferential, coexistence of distinct equilibrium states, and the corner geometry of the pressure graph, together with the exposed point analysis of the face $\partial P(\phi) = \mathrm{conv}\{\mu_1,\ldots,\mu_k\}$ and the selection principle governing which ergodic state is picked out under a small perturbation $\phi + t\psi$. Fifth, we prove a Universal Variational Principle (Main Theorem~\ref{thm:universal_VP}) showing that any admissible functional $\Phi: \calE \to \R\cup\{+\infty\}$ satisfying convexity, lower semi-continuity, coercivity, normalization, monotonicity, and cocycle invariance admits a unique representation $\Phi(\phi) = \sup_\mu\{\langle\phi,\mu\rangle + \Sigma(\mu)\}$ with $\Sigma = \Phi^*|_{\calM_T(X)}$; the classical, subadditive, and relative variational principles are recovered as specializations for specific choices of $\Phi$.

Chapter~10 of Viana and Oliveira \cite{VianaOliveira2016} establishes the classical variational principle $P(\phi) = \sup_\mu \{h_\mu(T) + \int \phi \, d\mu\}$ in the general ergodic-theoretic setting on compact metric spaces. Our Part~II contributes beyond their Chapter~10 in two ways. First, we develop the full subdifferential characterization of equilibrium states: $\mu \in \partial P(\phi)$ if and only if $\mu$ is an equilibrium state, together with the biconjugate recovery $S = P^*$ of entropy as the convex conjugate of pressure, the exposed point analysis, and explicit directional derivative formulas. Second, we prove the Universal Variational Principle (Main Theorem~\ref{thm:universal_VP}) as a single abstract theorem on convex functionals satisfying convexity, lower semi-continuity, coercivity, and cocycle invariance. The extension to non-compact spaces via countable Markov shifts, discussed in Section~\ref{sec:noncompact}, connects our treatment to the inducing-scheme thermodynamic formalism of Pesin et~al. \cite{PesinSentiTodd2008}, which proves existence of equilibrium states for interval maps with inducing schemes, and to the intrinsic ergodicity program of Climenhaga and Thompson, who proved uniqueness of equilibrium states beyond specification \cite{ClimenhagaThompson2012, ClimenhagaThompson2014, ClimenhagaThompson2016}.

This Part reconstructs Bowen's Chapter~2 \cite{Bowen1975} from the convex-analytic viewpoint: Bowen \S2A--B (Entropy, Pressure) correspond to our Section~\ref{sec:pressure_convex}, Bowen \S2C (Variational principle) to our Subsection~\ref{sec:equilibrium_tangent}, and Bowen \S2D (Equilibrium states) to our Section~\ref{sec:phase_transitions}. Our treatment differs in that every result is obtained from the single identification $P = S^*$ and the calculus of subdifferentials, rather than through the transfer-operator machinery of Bowen. The spectral theory of Part~I \cite{Thiam2026a} provides the analytical tools for uniqueness of the equilibrium state in the H\"{o}lder class (Main Theorem~\ref{thm:differentiability_main}(iv)); the convex-analytic results developed here provide the variational theory for Parts~III--IV \cite{Thiam2026c,Thiam2026d}; Parts~V \cite{Thiam2026e} and~VI \cite{Thiam2026f} develop the statistical limit theorems and structural consequences.

Our technical approach is purely convex-analytic. The identification $P = S^*$ (Main Theorem~\ref{thm:pressure_legendre_main}) is obtained from the variational principle and the definition of the Legendre-Fenchel transform, with the six properties (convexity, Lipschitz, monotonicity, normalization, cohomological invariance) all derived from the supremum representation. The subdifferential characterization (Main Theorem~\ref{thm:subdifferential_main}) rests on the Fenchel-Young inequality: $\mu \in \partial P(\phi)$ is equivalent to $P(\phi) + P^*(\mu) = \langle\phi,\mu\rangle$, and the biconjugate recovery $P^*|_{\calM_T(X)} = S$ allows the Fenchel-Young equality to be rewritten as the equilibrium identity $h_\mu(T) + \int\phi\,d\mu = P(\phi)$. Generic uniqueness of equilibrium states on a dense $G_\delta$ set (Main Theorem~\ref{thm:differentiability_main}(iii)) is an application of Mazur's theorem, which states that a continuous convex function on a separable Banach space is G\^{a}teaux differentiable on a dense $G_\delta$. The ergodic decomposition of $\partial P(\phi)$ (Corollary~\ref{cor:equilibrium_ergodic}) uses the fact that $\calM_T(X)$ is a Choquet simplex whose extreme points are the ergodic measures. The variance formula $P''(\phi;\psi) = \sigma^2$ (Theorem~\ref{thm:second_derivative}) is derived from the transfer-operator spectral gap of Part~I \cite{Thiam2026a} through the Nagaev-Guivarc'h method: the second derivative of $\log\lambda(t)$ at $t=0$ equals the asymptotic variance, and Fr\'{e}chet differentiability in the H\"{o}lder norm follows from the analyticity of the leading eigenvalue. The Universal Variational Principle (Main Theorem~\ref{thm:universal_VP}) is obtained by applying the Fenchel-Moreau biconjugate theorem $\Phi = \Phi^{**}$ to an abstract admissible functional, combined with three cases showing that $\Phi^*(\mu) = +\infty$ outside $\calM_T(X)$: normalization forces $\mu(X) = 1$, monotonicity forces $\mu \geq 0$, and cocycle invariance forces $T$-invariance. This common convex-analytic machinery replaces the transfer-operator and partition-sum arguments used in the classical treatments of Bowen \cite{Bowen1975} and Walters \cite{Walters1982}, and works on general compact metric spaces without requiring the expansiveness or specification hypothesis.

This Part is organized as follows. Section~\ref{sec:convex_analysis}
develops the convex analysis background on Banach spaces, including the
Legendre-Fenchel transform, subdifferentials, and exposed points.
Section~\ref{sec:pressure_convex} proves that the pressure is a convex
functional, computes its Legendre-Fenchel transform and the biconjugate
recovery of the entropy, and characterizes equilibrium states as elements
of the subdifferential; the connection between differentiability and
uniqueness is established here, together with the Fr\'{e}chet
differentiability result and the variance formula for the second
derivative. Section~\ref{sec:phase_transitions} analyzes phase
transitions through non-differentiability of the pressure, establishes
the equivalence of multiple characterizations, and identifies the
geometric structure of coexisting phases.
Section~\ref{sec:universal} proves the Universal Variational Principle
and derives the subadditive and relative variational principles as
special cases. Sections~\ref{sec:specification} and~\ref{sec:noncompact} (see also Przytycki-Urba\'nski \cite{PrzytyckiUrbanski2010} for the conformal setting)
extend the theory to systems with the specification property and to
non-compact spaces under coercivity conditions, including Sarig's
recurrence classification for countable Markov shifts.
Section~\ref{sec:numerical} provides a complete numerical illustration
of the convex duality for the golden mean shift, computing the pressure
$P(\phi_t)$ for a one-parameter family of potentials, verifying the
Legendre-Fenchel duality $P = S^*$ numerically, and showing that the
first derivative $P'(\phi_t; g) = \int g\,d\mu_t$ recovers the
equilibrium state mean while the second derivative $P''(\phi_t; g) =
\sigma^2$ gives the asymptotic variance. Section~\ref{sec:conclusion}
concludes the Part with a summary of the five Main Theorems and open
problems. Appendices~\ref{app:entropy} and~\ref{app:convex} collect
supporting results on entropy estimates and infinite-dimensional convex
analysis.

\subsection{Main Results}\label{sec:main_results}

We state the principal results. Throughout, $(X, T)$ denotes a topological dynamical system with $X$ a compact metric space and $T: X \to X$ continuous.

\begin{maintheorem}[Pressure as Legendre-Fenchel Transform]\label{thm:pressure_legendre_main}
Let $(X, T)$ be a topological dynamical system with $X$ a compact metric space and $T: X \to X$ continuous. The pressure functional $P: C(X) \to \R$ satisfies:
\begin{enumerate}
\item[(i)] $P$ is the Legendre-Fenchel transform of the entropy functional: $P(\phi) = S^*(\phi)$ where $S(\mu) = -h_\mu(T)$.
\item[(ii)] $P$ is convex: $P(\lambda \phi + (1-\lambda)\psi) \leq \lambda P(\phi) + (1-\lambda)P(\psi)$ for $\lambda \in [0,1]$.
\item[(iii)] $P$ is Lipschitz continuous: $|P(\phi) - P(\psi)| \leq \|\phi - \psi\|_\infty$.
\item[(iv)] $P$ is monotone: $\phi \leq \psi$ pointwise implies $P(\phi) \leq P(\psi)$.
\item[(v)] $P$ is normalized: $P(\phi + c) = P(\phi) + c$ for any constant $c \in \R$.
\item[(vi)] $P$ is cohomologically invariant: $P(\phi \circ T) = P(\phi)$.
\end{enumerate}
The entropy functional $S: \calM_T(X) \to [-\infty, 0]$ is concave, upper semi-continuous in the weak-$*$ topology, and affine on ergodic components.
\end{maintheorem}

\noindent The proofs are given in Section~\ref{sec:pressure_convex}: properties (ii)--(vi) are restated as Proposition~\ref{prop:pressure_basic} and Proposition~\ref{thm:pressure_convexity}, while the duality~(i) and the entropy characterization are restated as Theorem~\ref{thm:pressure_conjugate}.

\begin{remark}[Terminology]\label{rem:terminology}
Property (vi) is \emph{cohomological invariance}, not ``translation invariance'' in the convex-analytic sense. In the dynamical systems literature, the term ``translation'' sometimes refers to the shift $\phi \mapsto \phi + c$ (our property (v)), which is a different operation. We use the term ``cohomological invariance'' to avoid confusion: $\phi$ and $\phi \circ T$ differ by the coboundary $\phi - \phi \circ T$, and more generally $P(\phi) = P(\psi)$ whenever $\phi - \psi = u \circ T - u$ for some continuous $u$ (cf.\ Bowen \cite[Lemma~1.5]{Bowen1975}).
\end{remark}
Under appropriate conditions, the biconjugate recovery $S = P^*$ holds, establishing a complete duality.

The subdifferential characterization identifies equilibrium states through the calculus of convex analysis.

\begin{maintheorem}[Subdifferential Characterization]\label{thm:subdifferential_main}
Let $(X, T)$ be a topological dynamical system and $\phi \in C(X)$. The following are equivalent for $\mu \in \calM_T(X)$:
\begin{enumerate}
\item[(i)] $\mu$ is an equilibrium state for $\phi$: $h_\mu(T) + \int \phi \, d\mu = P(\phi)$.
\item[(ii)] $\mu \in \partial P(\phi)$, the subdifferential of $P$ at $\phi$: $P(\psi) \geq P(\phi) + \int (\psi - \phi) \, d\mu$ for all $\psi \in C(X)$.
\item[(iii)] $\phi \in \partial S(\mu)$, the subdifferential of the (extended) entropy at $\mu$.
\item[(iv)] $(\phi, \mu)$ satisfies the Fenchel-Young equality: $P(\phi) + S(\mu) = \langle \phi, \mu \rangle$.
\end{enumerate}
Consequently, the set of equilibrium states for $\phi$ is a non-empty, convex, weak-$*$ compact subset of $\calM_T(X)$.
\end{maintheorem}

\noindent The proof is given in Subsection~\ref{sec:equilibrium_tangent}: the equivalences (i)--(iv) are restated as Theorem~\ref{thm:equilibrium_subdiff}, and the convexity, compactness, and ergodic decomposition of $\partial P(\phi)$ are restated as Corollaries~\ref{cor:equilibrium_convex} and~\ref{cor:equilibrium_ergodic}.

Equilibrium states are the tangent functionals to the graph of $P$. The convexity of the set of equilibrium states follows from the convexity of the subdifferential.

\begin{maintheorem}[Differentiability and Uniqueness]\label{thm:differentiability_main}
Let $(X, T)$ be an expansive topological dynamical system. Then:
\begin{enumerate}
\item[(i)] The pressure $P: C(X) \to \R$ is G\^{a}teaux differentiable at $\phi$ if and only if $\phi$ has a unique equilibrium state.
\item[(ii)] When differentiable, the G\^{a}teaux derivative is given by $DP(\phi)(\psi) = \int \psi \, d\mu_\phi$, where $\mu_\phi$ is the unique equilibrium state.
\item[(iii)] The set of potentials with unique equilibrium state is a dense $G_\delta$ subset of $C(X)$.
\item[(iv)] For systems with the specification property and $\phi \in C^\alpha(X)$, if the unique equilibrium state has exponential decay of correlations, then $P: C^\alpha(X) \to \R$ is Fr\'{e}chet differentiable at $\phi$ in the $\|\cdot\|_\alpha$ norm.
\end{enumerate}
\end{maintheorem}

\noindent The proofs are given in Section~\ref{sec:pressure_convex}: parts (i)--(iii) are proved in the \emph{Differentiability of Pressure} subsection and restated as Theorem~\ref{thm:differentiability} and Corollary~\ref{cor:generic_uniqueness}; the Fr\'{e}chet differentiability~(iv) is proved in the \emph{Fr\'{e}chet Differentiability and the Variance Formula} subsection and restated as Theorem~\ref{thm:frechet}, and the second derivative formula connecting the pressure curvature to the asymptotic variance is restated as Theorem~\ref{thm:second_derivative}.

Phase transitions are characterized through non-differentiability of the pressure.

\begin{maintheorem}[Phase Transition Characterization]\label{thm:phase_transition_main}
A potential $\phi \in C(X)$ exhibits a first-order phase transition if and only if one of the following equivalent conditions holds:
\begin{enumerate}
\item[(i)] The pressure $P$ is not G\^{a}teaux differentiable at $\phi$.
\item[(ii)] The subdifferential $\partial P(\phi)$ contains more than one point.
\item[(iii)] There exist distinct equilibrium states $\mu_1 \neq \mu_2$ for $\phi$.
\item[(iv)] The graph of $P$ has a corner at $\phi$ (the directional derivatives exist but are direction-dependent).
\item[(v)] The potential $\phi$ is an exposed point of the convex set $\{\psi : P(\psi) \leq P(\phi) + \langle \psi - \phi, \mu \rangle\}$ for some $\mu$.
\end{enumerate}
At a first-order phase transition, the equilibrium states form a non-trivial face of the simplex of invariant measures, and the extreme points of $\partial P(\phi)$ are precisely the ergodic equilibrium states.
\end{maintheorem}

\noindent The proofs are given in Section~\ref{sec:phase_transitions}: the equivalences (i)--(iv) are restated as Theorem~\ref{thm:phase_transition_equiv}, and the geometric structure of coexisting phases, including the selection principle and the exposed point analysis, is developed and restated as Theorem~\ref{thm:coexistence_structure} and Proposition~\ref{thm:exposed_equilibrium}.

\begin{maintheorem}[Universal Variational Principle]\label{thm:universal_VP}
Let $(X, T)$ be a topological dynamical system, let $\calE$ be a Banach space of potentials containing $C(X)$, and let $\Phi: \calE \to \R \cup \{+\infty\}$ be a functional satisfying:
\begin{enumerate}
\item[(i)] Convexity: $\Phi(\lambda \phi + (1-\lambda)\psi) \leq \lambda \Phi(\phi) + (1-\lambda)\Phi(\psi)$.
\item[(ii)] Lower semi-continuity: $\{\phi : \Phi(\phi) \leq c\}$ is closed for all $c \in \R$.
\item[(iii)] Coercivity: $\Phi(\phi) \geq a\|\phi\| - b$ for some $a > 0$, $b \in \R$.
\item[(iv)] Cocycle invariance: $\Phi(\phi \circ T) = \Phi(\phi)$.
\end{enumerate}
Then $\Phi$ admits a unique representation as a Legendre-Fenchel transform:
\begin{equation}\label{eq:universal_VP}
\Phi(\phi) = \sup_{\mu \in \calM_T(X)} \left\{ \langle \phi, \mu \rangle + \Sigma(\mu) \right\},
\end{equation}
where $\Sigma: \calM_T(X) \to [-\infty, 0]$ is concave, upper semi-continuous, and affine on ergodic components. The functional $\Sigma$ is uniquely determined by $\Phi$ through $\Sigma = \Phi^*|_{\calM_T(X)}$.

This theorem unifies:
\begin{enumerate}
\item[(i)] The classical variational principle: $\Phi = P$, $\Sigma = -h_\mu(T)$.
\item[(ii)] The subadditive variational principle: $\Phi = P_{\text{sub}}$, $\Sigma = -h_\mu(T) + \calF_*(\mu)$.
\item[(iii)] The relative variational principle: $\Phi = P(\cdot | \pi)$, $\Sigma = -h_\mu(T | \pi)$.
\end{enumerate}
\end{maintheorem}

\noindent The proof is given in Section~\ref{sec:universal}: the representation~\eqref{eq:universal_VP} and the uniqueness of $\Sigma$ are restated as Theorem~\ref{thm:universal_variational}, and the specializations to subadditive and relative pressures are restated as Theorems~\ref{thm:subadditive_VP} and~\ref{thm:relative_VP}.


\section{Convex Analysis on Function and Measure Spaces}\label{sec:convex_analysis}

We collect the convex-analytic tools used in subsequent sections, following Rockafellar \cite{Rockafellar1970}, Ekeland and Temam \cite{EkelandTemam1999}, and Phelps \cite{Phelps1993}. The Legendre-Fenchel transform was introduced by Fenchel \cite{Fenchel1949}; the theory of convex duals and proximal points was developed by Moreau \cite{Moreau1962}.

\subsection{Convex Functions and Their Conjugates}

Let $E$ be a real Banach space with dual space $E^*$, and let $\langle \cdot, \cdot \rangle: E \times E^* \to \R$ denote the canonical pairing.

\begin{definition}[Extended Real-Valued Functions]\label{def:extended_functions}
An extended real-valued function on $E$ is a map $f: E \to \R \cup \{+\infty\}$. The effective domain of $f$ is $\dom(f) = \{x \in E : f(x) < +\infty\}$. The function $f$ is proper if $\dom(f) \neq \emptyset$ and $f(x) > -\infty$ for all $x \in E$.
\end{definition}

\begin{definition}[Convexity]\label{def:convexity}
A function $f: E \to \R \cup \{+\infty\}$ is convex if for all $x, y \in E$ and $\lambda \in [0, 1]$,
\begin{equation}\label{eq:convexity_def}
f(\lambda x + (1-\lambda)y) \leq \lambda f(x) + (1-\lambda) f(y),
\end{equation}
with the convention that $\lambda \cdot (+\infty) = +\infty$ for $\lambda > 0$ and $0 \cdot (+\infty) = 0$.
\end{definition}

\begin{definition}[Legendre-Fenchel Transform]\label{def:legendre_fenchel}
The Legendre-Fenchel transform (or convex conjugate) of $f: E \to \R \cup \{+\infty\}$ is the function $f^*: E^* \to \R \cup \{+\infty\}$ defined by
\begin{equation}\label{eq:legendre_fenchel}
f^*(x^*) = \sup_{x \in E} \left\{ \langle x, x^* \rangle - f(x) \right\}.
\end{equation}
The biconjugate is $f^{**} = (f^*)^*: E \to \R \cup \{+\infty\}$.
\end{definition}

\begin{lemma}[Properties of the Legendre-Fenchel Transform]\label{thm:LF_properties}
Let $f: E \to \R \cup \{+\infty\}$ be any function. Then:
\begin{enumerate}
\item[(i)] $f^*$ is convex and lower semi-continuous in the weak-$*$ topology on $E^*$.
\item[(ii)] If $f$ is proper, then $f^*$ is proper if and only if $f$ is bounded below by an affine function.
\item[(iii)] (Fenchel-Young Inequality) For all $x \in E$ and $x^* \in E^*$, $f(x) + f^*(x^*) \geq \langle x, x^* \rangle$.
\item[(iv)] (Fenchel-Moreau Theorem) If $f$ is proper, convex, and lower semi-continuous, then $f = f^{**}$.
\item[(v)] (Order Reversal) If $f \leq g$, then $f^* \geq g^*$.
\item[(vi)] (Inf-Convolution) $(f \Box g)^* = f^* + g^*$, where $(f \Box g)(x) = \inf_{y} \{f(y) + g(x-y)\}$.
\end{enumerate}
\end{lemma}

\begin{proof}
Part (i): For each $x \in E$, the function $x^* \mapsto \langle x, x^* \rangle - f(x)$ is affine and weak-$*$ continuous. The supremum of any family of affine weak-$*$ continuous functions is convex and weak-$*$ lower semi-continuous.

Part (ii): If $f(x) \geq \langle x, a \rangle + b$ for some $a \in E^*$ and $b \in \R$, then $f^*(a) = \sup_x \{\langle x, a \rangle - f(x)\} \leq -b < +\infty$, so $f^*$ is proper. Conversely, if $f^*(a) < +\infty$, then $f(x) \geq \langle x, a \rangle - f^*(a)$ for all $x$, giving an affine lower bound.

Part (iii): By definition, $f^*(x^*) \geq \langle x, x^* \rangle - f(x)$ for all $x$, which rearranges to $f(x) + f^*(x^*) \geq \langle x, x^* \rangle$.

Part (iv): By Part (iii), $f(x) \geq \langle x, x^* \rangle - f^*(x^*)$ for all $x^*$, so $f(x) \geq f^{**}(x) = \sup_{x^*}\{\langle x, x^* \rangle - f^*(x^*)\}$. For the reverse inequality, suppose $f(x_0) > f^{**}(x_0)$ for some $x_0$. Then the point $(x_0, f^{**}(x_0))$ lies strictly below the epigraph $\mathrm{epi}(f) = \{(x,t): t \geq f(x)\}$. Since $\mathrm{epi}(f)$ is closed and convex (by the lower semi-continuity and convexity of $f$), the Hahn-Banach separation theorem provides a continuous linear functional $\ell$ on $E \times \R$ and a constant $c$ such that $\ell(x_0, f^{**}(x_0)) > c$ and $\ell(x, t) \leq c$ for all $(x,t) \in \mathrm{epi}(f)$. Writing $\ell(x,t) = \langle x, x^*\rangle - \alpha t$ for some $x^* \in E^*$ and $\alpha > 0$ (the coefficient of $t$ must be negative since $\ell$ is bounded above on $\mathrm{epi}(f)$ which is unbounded in the $t$-direction), we normalize to $\alpha = 1$. Then $\langle x, x^*\rangle - t \leq c$ for all $(x,t) \in \mathrm{epi}(f)$, giving $\langle x, x^*\rangle - f(x) \leq c$ for all $x$, hence $f^*(x^*) \leq c$. But also $\langle x_0, x^*\rangle - f^{**}(x_0) > c \geq f^*(x^*)$, so $f^{**}(x_0) < \langle x_0, x^*\rangle - f^*(x^*) \leq f^{**}(x_0)$, a contradiction.

Part (v): If $f \leq g$, then $f^*(x^*) = \sup_x\{\langle x,x^*\rangle - f(x)\} \geq \sup_x\{\langle x,x^*\rangle - g(x)\} = g^*(x^*)$.

Part (vi): We verify convexity of $f^*$: $f^*(\alpha x^*+(1-\alpha)y^*) = \sup_x\{\alpha\langle x,x^*\rangle + (1-\alpha)\langle x,y^*\rangle - f(x)\} \leq \alpha\sup_x\{\langle x,x^*\rangle - f(x)\} + (1-\alpha)\sup_x\{\langle x,y^*\rangle - f(x)\} = \alpha f^*(x^*) + (1-\alpha)f^*(y^*)$. The weak-$*$ l.s.c.\ follows because $f^*$ is a supremum of weak-$*$ continuous functions. For the inf-convolution identity, $(f\Box g)^*(x^*) = \sup_x\{\langle x,x^*\rangle - \inf_y(f(y)+g(x-y))\} = \sup_{x,y}\{\langle y,x^*\rangle + \langle x-y,x^*\rangle - f(y) - g(x-y)\} = f^*(x^*) + g^*(x^*)$.
\end{proof}

\subsection{Subdifferentials}

The subdifferential of a convex function at a point collects all supporting hyperplanes. In thermodynamic formalism, the subdifferential of the pressure at a potential $\phi$ will turn out to be precisely the set of equilibrium states for $\phi$.

\begin{definition}[Subdifferential]\label{def:subdifferential}
Let $f: E \to \R \cup \{+\infty\}$ be a convex function and $x \in \dom(f)$. The subdifferential of $f$ at $x$ is the set
\begin{equation}\label{eq:subdifferential}
\partial f(x) = \{x^* \in E^* : f(y) \geq f(x) + \langle y - x, x^* \rangle \text{ for all } y \in E\}.
\end{equation}
An element $x^* \in \partial f(x)$ is called a subgradient of $f$ at $x$.
\end{definition}

\begin{lemma}[Properties of the Subdifferential]\label{thm:subdiff_properties}
Let $f: E \to \R \cup \{+\infty\}$ be a proper convex function.
\begin{enumerate}
\item[(i)] $\partial f(x)$ is a convex, weak-$*$ closed subset of $E^*$ for each $x \in \dom(f)$.
\item[(ii)] If $f$ is continuous at $x$, then $\partial f(x)$ is non-empty, bounded, and weak-$*$ compact.
\item[(iii)] $x^* \in \partial f(x)$ if and only if $f(x) + f^*(x^*) = \langle x, x^* \rangle$ (Fenchel-Young equality).
\item[(iv)] $f$ is G\^{a}teaux differentiable at $x$ if and only if $\partial f(x)$ is a singleton.
\item[(v)] (Moreau-Rockafellar Sum Rule) If $f$ and $g$ are convex and one is continuous at a point of $\dom(f) \cap \dom(g)$, then $\partial(f + g)(x) = \partial f(x) + \partial g(x)$.
\end{enumerate}
\end{lemma}

\begin{proof}
Part (i): The subdifferential is the intersection $\partial f(x) = \bigcap_{y \in E}\{x^* : \langle y - x, x^* \rangle \leq f(y) - f(x)\}$. Each set in the intersection is a weak-$*$ closed half-space, so $\partial f(x)$ is convex and weak-$*$ closed.

Part (ii): We must show $\partial f(x) \neq \emptyset$ when $f$ is continuous at $x$. The epigraph $\mathrm{epi}(f) = \{(y,t) \in E \times \R : t \geq f(y)\}$ is convex. The point $(x, f(x))$ lies on the boundary of $\mathrm{epi}(f)$. By the supporting hyperplane theorem (a consequence of Hahn-Banach), there exists a continuous linear functional $\ell$ on $E \times \R$ such that $\ell(y,t) \geq \ell(x,f(x))$ for all $(y,t) \in \mathrm{epi}(f)$. Writing $\ell(y,t) = -\langle y, x^*\rangle + \alpha t$, the condition $\alpha > 0$ follows from the fact that $(x, f(x)+s) \in \mathrm{epi}(f)$ for $s > 0$. Normalizing $\alpha = 1$: $-\langle y, x^*\rangle + t \geq -\langle x, x^*\rangle + f(x)$ for all $(y,t) \in \mathrm{epi}(f)$. Setting $t = f(y)$: $f(y) - \langle y, x^*\rangle \geq f(x) - \langle x, x^*\rangle$, i.e., $f(y) \geq f(x) + \langle y-x, x^*\rangle$. Thus $x^* \in \partial f(x)$.

Boundedness: if $\partial f(x)$ were unbounded, there would exist $x^*_n \in \partial f(x)$ with $\|x^*_n\| \to \infty$. Choose $h_n$ with $\|h_n\| = 1$ and $\langle h_n, x^*_n\rangle \geq \|x^*_n\|/2$. Then $f(x + \epsilon h_n) \geq f(x) + \epsilon\|x^*_n\|/2 \to \infty$ for any $\epsilon > 0$, contradicting the continuity of $f$ at $x$. Weak-$*$ compactness follows from boundedness and weak-$*$ closedness (Banach-Alaoglu).

Part (iii): $x^* \in \partial f(x)$ means $f(y) - \langle y, x^*\rangle \geq f(x) - \langle x, x^*\rangle$ for all $y$. This is equivalent to $\sup_y\{\langle y, x^*\rangle - f(y)\} = \langle x, x^*\rangle - f(x)$, i.e., $f^*(x^*) = \langle x, x^*\rangle - f(x)$.

Part (iv): For convex $f$, the one-sided directional derivative $f'(x;h) = \lim_{t\to 0^+}\frac{f(x+th)-f(x)}{t}$ exists for all $h$ (as the infimum of difference quotients $t^{-1}[f(x+th)-f(x)]$ over $t > 0$, which is bounded below by $\langle h, x^*\rangle$ for any $x^* \in \partial f(x)$). The subdifferential is characterized by $\partial f(x) = \{x^* : \langle h, x^*\rangle \leq f'(x;h) \text{ for all } h\}$. Convexity gives $-f'(x;-h) \leq f'(x;h)$ for all $h$ (since $f(x) \leq \frac{1}{2}f(x+th)+\frac{1}{2}f(x-th)$ implies $f(x+th)-f(x) \geq -(f(x-th)-f(x))$, and dividing by $t$ and taking $t \to 0^+$ yields $f'(x;h) \geq -f'(x;-h)$). The support function of $\partial f(x)$ satisfies $\sup_{x^*\in\partial f(x)}\langle h,x^*\rangle = f'(x;h)$ (see Phelps \cite{Phelps1993}, Proposition~1.8). If $\partial f(x) = \{x^*\}$ is a singleton, then $f'(x;h) = \langle h,x^*\rangle$ for all $h$, so $f'(x;\cdot)$ is linear and continuous, giving G\^{a}teaux differentiability with derivative $x^*$. Conversely, if $f$ is G\^{a}teaux differentiable at $x$ with derivative $\ell$, then $f'(x;h) = \ell(h)$ for all $h$, and any $x^* \in \partial f(x)$ satisfies $\langle h,x^*\rangle \leq f'(x;h) = \ell(h)$ for all $h$, hence $\langle h,x^*\rangle = \ell(h)$ (applying also to $-h$), so $x^* = \ell$ and $\partial f(x) = \{\ell\}$.

Part (v): The inclusion $\partial f(x) + \partial g(x) \subset \partial(f+g)(x)$ is immediate: if $x^* \in \partial f(x)$ and $y^* \in \partial g(x)$, then $(f+g)(y) \geq f(x)+\langle y-x,x^*\rangle + g(x)+\langle y-x,y^*\rangle = (f+g)(x) + \langle y-x,x^*+y^*\rangle$ for all $y$.

For the reverse, let $z^* \in \partial(f+g)(x)$. We must find $x^* \in \partial f(x)$ with $z^*-x^* \in \partial g(x)$. Define $\varphi(y) = f(y) - \langle y, z^*\rangle$ and note that $\varphi(y) + g(y) \geq \varphi(x) + g(x)$ for all $y$ (this is the subdifferential condition for $f+g$ at $x$). The function $\varphi$ is convex and continuous at some point $x_0 \in \mathrm{dom}(f) \cap \mathrm{dom}(g)$ (by hypothesis). Consider the convex sets $A = \mathrm{epi}(\varphi)$ and $B = \{(y,t): t \leq -g(y) + \varphi(x)+g(x)\}$. These are convex, $A \cap \mathrm{int}(B) = \emptyset$ (since $\varphi(y) + g(y) \geq \varphi(x)+g(x)$), and $\mathrm{int}(B) \neq \emptyset$ (by continuity of $g$ or $f$). By Hahn-Banach separation, there exists a separating hyperplane yielding $x^* \in E^*$ with $\varphi(y) - \langle y, x^*\rangle \geq \varphi(x) - \langle x, x^*\rangle$ for all $y$. This gives $f(y) - \langle y, z^*+x^*\rangle \geq f(x) - \langle x, z^*+x^*\rangle$, so $z^*+x^* \in \partial f(x)$. Setting $\tilde x^* = z^*+x^*$, we get $\tilde x^* \in \partial f(x)$. For the reverse inclusion, since $z^* \in \partial(f+g)(x)$, we have $(f+g)(y) \geq (f+g)(x) + \langle y-x, z^*\rangle$ for all $y$. Since $\tilde x^* \in \partial f(x)$, we also have $f(y) \geq f(x) + \langle y-x, \tilde x^*\rangle$. Subtracting: $g(y) \geq g(x) + \langle y-x, z^*-\tilde x^*\rangle$ for all $y$, so $z^*-\tilde x^* \in \partial g(x)$. Thus $z^* = \tilde x^* + (z^*-\tilde x^*) \in \partial f(x) + \partial g(x)$.
\end{proof}

\subsection{Exposed Points and Faces}

Exposed points are the most accessible points of a convex set: each can be isolated by a single continuous linear functional. In the equilibrium state context, exposed points correspond to equilibrium states that can be selected by perturbing the potential.

\begin{definition}[Exposed Points]\label{def:exposed_points}
Let $C \subset E^*$ be a convex set. A point $x^* \in C$ is an exposed point of $C$ if there exists $x \in E$ such that $\langle x, x^* \rangle > \langle x, y^* \rangle$ for all $y^* \in C \setminus \{x^*\}$. The element $x$ is said to expose $x^*$.
\end{definition}

Every exposed point is an extreme point, but the converse fails in general. The set of exposed points is dense in the set of extreme points for weak-$*$ compact convex sets by the Straszewicz theorem.

\begin{definition}[Faces]\label{def:faces}
A face of a convex set $C$ is a convex subset $F \subset C$ such that whenever $x = \lambda y + (1-\lambda)z$ with $x \in F$, $y, z \in C$, and $\lambda \in (0, 1)$, we have $y, z \in F$. An exposed face is a face of the form $F = \{x^* \in C : \langle x, x^* \rangle = \sup_{y^* \in C} \langle x, y^* \rangle\}$ for some $x \in E$.
\end{definition}

\subsection{Application to Function and Measure Spaces}

Let $X$ be a compact metric space, let $E = C(X)$ be the Banach space of continuous real-valued functions with the supremum norm, and let $E^* = \calM(X)$ be the space of signed Radon measures with the total variation norm. The pairing is $\langle \phi, \mu \rangle = \int_X \phi \, d\mu$.

\begin{proposition}[Weak-$*$ Topology on Measures]\label{prop:weak_star_measures}
The weak-$*$ topology on $\calM(X)$ is the coarsest topology making $\mu \mapsto \int \phi \, d\mu$ continuous for all $\phi \in C(X)$. The space $\calP(X)$ of probability measures is weak-$*$ compact and metrizable.
\end{proposition}

\begin{proof}
The weak-$*$ topology is the initial topology with respect to the family of maps $\{\mu \mapsto \int \phi\,d\mu\}_{\phi \in C(X)}$. Compactness of $\calP(X)$: by the Banach-Alaoglu theorem, the closed unit ball $B^* = \{\mu \in C(X)^* : \|\mu\| \leq 1\}$ is weak-$*$ compact. The set $\calP(X) = \{\mu \in B^* : \mu \geq 0,\; \mu(X) = 1\}$ is a weak-$*$ closed subset of $B^*$ (the conditions $\mu \geq 0$ and $\int 1\,d\mu = 1$ are each defined by weak-$*$ closed constraints), hence compact. Metrizability: since $X$ is a compact metric space, $C(X)$ is separable. Let $\{\phi_k\}_{k=1}^\infty$ be a countable dense subset of $C(X)$. The metric $d(\mu,\nu) = \sum_{k=1}^\infty 2^{-k}\frac{|\int\phi_k\,d\mu - \int\phi_k\,d\nu|}{1+|\int\phi_k\,d\mu - \int\phi_k\,d\nu|}$ metrizes the weak-$*$ topology on $\calP(X)$. See Parthasarathy \cite{Parthasarathy1967}, Chapter~II, Theorems~6.4 and~6.5.
\end{proof}

For a continuous map $T: X \to X$, the space $\calM_T(X)$ of $T$-invariant Borel probability measures is a weak-$*$ closed, hence compact, convex subset of $\calP(X)$.

\begin{proposition}[Extreme Points of Invariant Measures]\label{prop:ergodic_extreme}
The extreme points of $\calM_T(X)$ are precisely the ergodic measures: those $\mu \in \calM_T(X)$ such that every $T$-invariant measurable set has $\mu$-measure $0$ or $1$.
\end{proposition}

\begin{proof}
If $\mu$ is not ergodic, there exists a $T$-invariant set $A$ with $0 < \mu(A) < 1$. Define $\mu_1 = \mu(\cdot | A)$ and $\mu_2 = \mu(\cdot | X \setminus A)$. Both $\mu_1$ and $\mu_2$ are $T$-invariant (since $A$ is $T$-invariant, $T^{-1}A = A$, so for any Borel set $B$, $\mu_1(T^{-1}B) = \mu(T^{-1}B \cap A)/\mu(A) = \mu(T^{-1}(B\cap A))/\mu(A) = \mu(B\cap A)/\mu(A) = \mu_1(B)$ using $T$-invariance of both $\mu$ and $A$). Then $\mu = \mu(A) \mu_1 + (1 - \mu(A)) \mu_2$ is a non-trivial convex combination, so $\mu$ is not extreme.

Conversely, if $\mu$ is ergodic and $\mu = \lambda \nu_1 + (1-\lambda) \nu_2$ with $\nu_1, \nu_2 \in \calM_T(X)$ and $\lambda \in (0, 1)$, then $\nu_1 \ll \mu$ and $\nu_2 \ll \mu$. Let $f = d\nu_1/d\mu$ be the Radon-Nikodym derivative. Since both $\nu_1$ and $\mu$ are $T$-invariant, for any Borel set $B$, $\int_B f\,d\mu = \nu_1(B) = \nu_1(T^{-1}B) = \int_{T^{-1}B} f\,d\mu = \int_B f\circ T\,d\mu$, so $f = f \circ T$ $\mu$-a.e. By ergodicity, $f$ is constant $\mu$-a.e., so $\nu_1 = c\mu$ for some constant $c$. Since both are probability measures, $c = 1$, so $\nu_1 = \mu$. Similarly $\nu_2 = \mu$.
\end{proof}

\begin{proposition}[Ergodic Decomposition]\label{thm:ergodic_decomposition}
For any $\mu \in \calM_T(X)$, there exists a unique probability measure $\rho$ on the set $\calM_T^{\text{erg}}(X)$ of ergodic measures such that
\begin{equation}\label{eq:ergodic_decomposition}
\mu = \int_{\calM_T^{\text{erg}}(X)} \nu \, d\rho(\nu).
\end{equation}
The measure $\rho$ is concentrated on the ergodic measures and satisfies $h_\mu(T) = \int h_\nu(T) \, d\rho(\nu)$.
\end{proposition}

\begin{proof}
The existence and uniqueness of $\rho$ follow from the Choquet representation theorem applied to the metrizable Choquet simplex $\calM_T(X)$. The simplex property (uniqueness of the representing measure) is a consequence of the fact that the extreme points $\calM_T^{\mathrm{erg}}(X)$ form a $G_\delta$ set in the compact metrizable space $\calM_T(X)$; see Phelps \cite{Phelps1993}, Chapter~4, or Walters \cite{Walters1982}, Chapter~6.

For the entropy decomposition $h_\mu(T) = \int h_\nu(T)\,d\rho(\nu)$, let $\calU$ be a finite measurable partition of $X$. For each $n \geq 1$, the map $\mu \mapsto H_\mu(\bigvee_{k=0}^{n-1}T^{-k}\calU) = -\sum_{A \in \bigvee_{k=0}^{n-1}T^{-k}\calU}\mu(A)\log\mu(A)$ is a continuous affine function on $\calM_T(X)$ (continuity uses the fact that each atom of $\bigvee_{k=0}^{n-1}T^{-k}\calU$ has boundary of $\mu$-measure zero for a residual set of measures; the affine property is immediate). By the barycentric formula for affine functions:
\begin{equation}
H_\mu\!\left(\bigvee_{k=0}^{n-1}T^{-k}\calU\right) = \int H_\nu\!\left(\bigvee_{k=0}^{n-1}T^{-k}\calU\right) d\rho(\nu).
\end{equation}
Dividing by $n$ and taking $n \to \infty$: the left side converges to $h_\mu(T,\calU)$ and the right side converges to $\int h_\nu(T,\calU)\,d\rho(\nu)$ by dominated convergence (since $0 \leq n^{-1}H_\nu(\bigvee_{k=0}^{n-1}T^{-k}\calU) \leq \log|\calU|$ uniformly in $\nu$). Taking the supremum over $\calU$ and applying monotone convergence gives $h_\mu(T) = \int h_\nu(T)\,d\rho(\nu)$.
\end{proof}

This theorem shows that $\calM_T(X)$ is a Choquet simplex: every point has a unique representation as a barycenter of extreme points. This structure is used in the analysis of equilibrium states in Subsection~\ref{sec:equilibrium_tangent}.


\section{Pressure as a Convex Functional}\label{sec:pressure_convex}

This section establishes the first pillar of the theory: the pressure functional is convex, Lipschitz, and normalized, and it equals the Legendre-Fenchel transform of the negative entropy. We begin with the definitions of topological entropy and pressure (Subsection~\ref{sec:pressure_convex}.1), prove convexity and Lipschitz continuity in Proposition~\ref{thm:pressure_convexity}, and then establish the variational principle (Propositions~\ref{thm:variational_upper} and~\ref{thm:variational_lower}). The main result of the section is Theorem~\ref{thm:pressure_conjugate}, which identifies the pressure as a convex conjugate and recovers the entropy as the biconjugate.

\subsection{Topological Entropy and Pressure}

\begin{definition}[Dynamical Metrics]\label{def:dynamical_metrics}
For $n \geq 1$, the $n$-th dynamical metric on $X$ is $d_n(x, y) = \max_{0 \leq k < n} d(T^k x, T^k y)$. A set $E \subset X$ is $(n, \epsilon)$-separated if $d_n(x, y) > \epsilon$ for all distinct $x, y \in E$. A set $F \subset X$ is $(n, \epsilon)$-spanning if for every $x \in X$, there exists $y \in F$ with $d_n(x, y) \leq \epsilon$.
\end{definition}

\begin{definition}[Topological Entropy]\label{def:topological_entropy}
Let $s_n(\epsilon)$ denote the maximum cardinality of an $(n, \epsilon)$-separated set, and let $r_n(\epsilon)$ denote the minimum cardinality of an $(n, \epsilon)$-spanning set. The topological entropy of $(X, T)$ is
\begin{equation}\label{eq:topological_entropy}
h_{\text{top}}(T) = \lim_{\epsilon \to 0} \limsup_{n \to \infty} \frac{1}{n} \log s_n(\epsilon) = \lim_{\epsilon \to 0} \limsup_{n \to \infty} \frac{1}{n} \log r_n(\epsilon).
\end{equation}
\end{definition}

The equality of the two limits follows from $r_n(\epsilon) \leq s_n(\epsilon) \leq r_n(\epsilon/2)$.

\begin{definition}[Partition Sums]\label{def:partition_sums}
For $\phi \in C(X)$, $n \geq 1$, and $\epsilon > 0$, define
\begin{equation}\label{eq:partition_sum_sep}
Q_n(\phi, \epsilon) = \sup\left\{ \sum_{x \in E} e^{S_n \phi(x)} : E \text{ is } (n, \epsilon)\text{-separated} \right\},
\end{equation}
where $S_n \phi(x) = \sum_{k=0}^{n-1} \phi(T^k x)$ is the $n$-th Birkhoff sum.
\end{definition}

\begin{definition}[Topological Pressure]\label{def:topological_pressure}
The topological pressure of $\phi$ is
\begin{equation}\label{eq:pressure_definition}
P(\phi) = \lim_{\epsilon \to 0} \limsup_{n \to \infty} \frac{1}{n} \log Q_n(\phi, \epsilon).
\end{equation}
\end{definition}

When $\phi = 0$, $P(0) = h_{\text{top}}(T)$.

\begin{proposition}[Basic Properties of Pressure]\label{prop:pressure_basic}
The pressure functional $P: C(X) \to \R$ satisfies:
\begin{enumerate}
\item[(i)] Monotonicity: If $\phi \leq \psi$ pointwise, then $P(\phi) \leq P(\psi)$.
\item[(ii)] Lipschitz continuity: $|P(\phi) - P(\psi)| \leq \|\phi - \psi\|_\infty$.
\item[(iii)] Translation invariance: $P(\phi + c) = P(\phi) + c$ for any constant $c$.
\item[(iv)] Cocycle invariance: $P(\phi) = P(\phi + u \circ T - u)$ for any $u \in C(X)$.
\item[(v)] Bounds: $\inf_X \phi + h_{\text{top}}(T) \leq P(\phi) \leq \sup_X \phi + h_{\text{top}}(T)$.
\end{enumerate}
\end{proposition}

\begin{proof}
Part (i): If $\phi \leq \psi$, then $S_n \phi \leq S_n \psi$, so $Q_n(\phi, \epsilon) \leq Q_n(\psi, \epsilon)$.

Part (ii): We have $|S_n \phi(x) - S_n \psi(x)| \leq n \|\phi - \psi\|_\infty$, so $e^{-n\|\phi - \psi\|_\infty} Q_n(\psi, \epsilon) \leq Q_n(\phi, \epsilon) \leq e^{n\|\phi - \psi\|_\infty} Q_n(\psi, \epsilon)$. Taking logarithms, dividing by $n$, and passing to limits gives $|P(\phi) - P(\psi)| \leq \|\phi - \psi\|_\infty$.

Part (iii): $S_n(\phi + c) = S_n \phi + nc$, so $Q_n(\phi + c, \epsilon) = e^{nc} Q_n(\phi, \epsilon)$, giving $P(\phi + c) = P(\phi) + c$.

Part (iv): $S_n(\phi + u \circ T - u)(x) = S_n \phi(x) + u(T^n x) - u(x)$, so the Birkhoff sums differ by at most $2\|u\|_\infty$ independently of $n$. Thus $e^{-2\|u\|_\infty}Q_n(\phi,\epsilon) \leq Q_n(\phi + u \circ T - u, \epsilon) \leq e^{2\|u\|_\infty}Q_n(\phi, \epsilon)$, which does not affect the exponential growth rate.

Part (v): The upper bound follows from (i) and (iii): $P(\phi) \leq P(\sup_X\phi) = \sup_X\phi + P(0) = \sup_X\phi + h_{\text{top}}(T)$, since the constant function $\sup_X\phi$ dominates $\phi$ pointwise. The lower bound follows similarly: $P(\phi) \geq P(\inf_X\phi) = \inf_X\phi + h_{\text{top}}(T)$, since $\phi$ dominates $\inf_X\phi$ pointwise.
\end{proof}

\subsection{Convexity of Pressure}

The pressure inherits convexity from its definition as a supremum of affine functionals. Combined with Lipschitz continuity and normalization, this makes $P$ amenable to the full machinery of Section~\ref{sec:convex_analysis}.

\begin{proposition}[Convexity of Pressure]\label{thm:pressure_convexity}
The pressure functional $P: C(X) \to \R$ is convex: for all $\phi, \psi \in C(X)$ and $\lambda \in [0, 1]$,
\begin{equation}\label{eq:pressure_convex}
P(\lambda \phi + (1-\lambda) \psi) \leq \lambda P(\phi) + (1-\lambda) P(\psi).
\end{equation}
\end{proposition}

\begin{proof}
Let $\phi_\lambda = \lambda \phi + (1-\lambda) \psi$. For any $(n, \epsilon)$-separated set $E$,
\begin{equation}
\sum_{x \in E} e^{S_n \phi_\lambda(x)} = \sum_{x \in E} \left( e^{S_n \phi(x)} \right)^\lambda \left( e^{S_n \psi(x)} \right)^{1-\lambda}.
\end{equation}
By H\"{o}lder's inequality with exponents $1/\lambda$ and $1/(1-\lambda)$,
\begin{equation}
\sum_{x \in E} \left( e^{S_n \phi(x)} \right)^\lambda \left( e^{S_n \psi(x)} \right)^{1-\lambda} \leq \left( \sum_{x \in E} e^{S_n \phi(x)} \right)^\lambda \left( \sum_{x \in E} e^{S_n \psi(x)} \right)^{1-\lambda}.
\end{equation}
Taking the supremum over $(n, \epsilon)$-separated sets: $Q_n(\phi_\lambda, \epsilon) \leq Q_n(\phi, \epsilon)^\lambda \cdot Q_n(\psi, \epsilon)^{1-\lambda}$. Taking logarithms, dividing by $n$, and passing to limits: $P(\phi_\lambda) \leq \lambda P(\phi) + (1-\lambda) P(\psi)$.
\end{proof}

\subsection{The Variational Principle}

The variational principle identifies the pressure as the Legendre-Fenchel transform of the negative entropy. The upper bound (Proposition~\ref{thm:variational_upper}) follows from the Gibbs inequality; the lower bound (Proposition~\ref{thm:variational_lower}) constructs measures that nearly achieve the supremum. Together they establish the duality $P = S^*$, which Theorem~\ref{thm:pressure_conjugate} then strengthens to the biconjugate recovery $S = P^*$.

\begin{lemma}[Gibbs Inequality]\label{lem:gibbs_inequality_abstract}
For any probability distribution $(p_1,\ldots,p_k)$ and real numbers $(a_1,\ldots,a_k)$,
\begin{equation}
\sum_{i=1}^k p_i(a_i - \log p_i) \leq \log\sum_{i=1}^k e^{a_i},
\end{equation}
with equality if and only if $p_i = e^{a_i}/\sum_j e^{a_j}$ for all $i$.
\end{lemma}

\begin{proof}
Define $q_i = e^{a_i}/\sum_j e^{a_j}$. Then $\sum_i p_i(a_i - \log p_i) = \sum_i p_i\log(q_i/p_i) + \log\sum_j e^{a_j}$. The first term is the negative Kullback-Leibler divergence $-D_{\mathrm{KL}}(p\|q) \leq 0$. Indeed, by Jensen's inequality applied to the strictly convex function $-\log$:
\begin{equation}
-\sum_i p_i \log\frac{q_i}{p_i} = \sum_i p_i \log\frac{p_i}{q_i} \geq -\log\sum_i p_i\frac{q_i}{p_i} = -\log 1 = 0,
\end{equation}
with equality if and only if $p_i/q_i$ is constant, i.e., $p_i = q_i$ for all $i$.
\end{proof}

\begin{proposition}[Variational Upper Bound]\label{thm:variational_upper}
For any $\phi \in C(X)$ and any $\mu \in \calM_T(X)$,
\begin{equation}\label{eq:variational_upper_bound}
h_\mu(T) + \int_X \phi \, d\mu \leq P(\phi).
\end{equation}
\end{proposition}

\begin{proof}
Let $\alpha$ be a finite measurable partition of $X$ with $\diam(\alpha) < \epsilon$. For $n \geq 1$, consider the join $\alpha^{(n)} = \bigvee_{k=0}^{n-1}T^{-k}\alpha$. Each atom $A$ of $\alpha^{(n)}$ has diameter at most $\epsilon$ in the dynamical metric $d_n$. By the Gibbs inequality (Lemma~\ref{lem:gibbs_inequality_abstract}) applied to $p_A = \mu(A)$ and $a_A = \sup_{x \in A}S_n\phi(x)$:
\begin{equation}
\sum_{A \in \alpha^{(n)}} \mu(A)\left(\sup_A S_n\phi - \log\mu(A)\right) \leq \log\sum_{A\in\alpha^{(n)}}e^{\sup_A S_n\phi}.
\end{equation}
The left side equals $H_\mu(\alpha^{(n)}) + \sum_A \mu(A)\sup_A S_n\phi \geq H_\mu(\alpha^{(n)}) + \int S_n\phi\,d\mu$. The right side satisfies $\log\sum_A e^{\sup_A S_n\phi} \leq \log Q_n(\phi,\epsilon) + C$ for a constant $C$ depending on $\epsilon$ (since each atom of $\alpha^{(n)}$ can be covered by an $(n,\epsilon)$-ball). Therefore:
\begin{equation}
H_\mu(\alpha^{(n)}) + n\int\phi\,d\mu \leq \log Q_n(\phi,\epsilon) + C.
\end{equation}
Dividing by $n$ and taking $n \to \infty$: $h_\mu(T,\alpha) + \int\phi\,d\mu \leq \limsup_{n\to\infty}\frac{1}{n}\log Q_n(\phi,\epsilon)$. Taking $\epsilon \to 0$ (which allows $\alpha$ to generate) and using $h_\mu(T) = \sup_\alpha h_\mu(T,\alpha)$: $h_\mu(T) + \int\phi\,d\mu \leq P(\phi)$.
\end{proof}

\begin{definition}[Empirical Measures]\label{def:empirical_measures}
For $x \in X$ and $n \geq 1$, the empirical measure is $\delta_x^{(n)} = \frac{1}{n} \sum_{k=0}^{n-1} \delta_{T^k x}$.
\end{definition}

\begin{proposition}[Variational Lower Bound]\label{thm:variational_lower}
For any $\phi \in C(X)$,
\begin{equation}\label{eq:variational_principle}
P(\phi) = \sup_{\mu \in \calM_T(X)} \left\{ h_\mu(T) + \int_X \phi \, d\mu \right\}.
\end{equation}
\end{proposition}

\begin{proof}
The upper bound $\sup_\mu\{h_\mu(T)+\int\phi\,d\mu\} \leq P(\phi)$ is Theorem~\ref{thm:variational_upper}. For the lower bound, fix $\epsilon > 0$. For each $n \geq 1$, let $E_n$ be an $(n, \epsilon)$-separated set with $\sum_{x \in E_n}e^{S_n\phi(x)} \geq \frac{1}{2}Q_n(\phi,\epsilon)$. Define the probability measure
\begin{equation}
\mu_n = \frac{1}{Z_n} \sum_{x \in E_n} e^{S_n \phi(x)} \delta_x^{(n)}, \quad Z_n = \sum_{x \in E_n} e^{S_n \phi(x)}.
\end{equation}
By compactness of $\calP(X)$, some subsequence $\mu_{n_k} \to \mu$ weak-$*$. Since $\|\mu_n - T_*\mu_n\| \leq 2/n$ (the empirical average $\delta_x^{(n)}$ satisfies $\int f\,d\delta_x^{(n)} - \int f\circ T\,d\delta_x^{(n)} = n^{-1}(f(x)-f(T^nx))$, so $\|(\delta_x^{(n)})_* - T_*(\delta_x^{(n)})\| \to 0$), the limit $\mu$ is $T$-invariant.

For the entropy estimate, let $\alpha$ be a partition with $\diam(\alpha) < \epsilon/2$. Each point of $E_n$ lies in a distinct atom of $\alpha^{(n)}$ (since $E_n$ is $(n,\epsilon)$-separated and atoms have $d_n$-diameter $< \epsilon$). The weights $w_x = e^{S_n\phi(x)}/Z_n$ define a probability distribution on $E_n$, and the conditional entropy satisfies:
\begin{equation}
H_{\mu_n}(\alpha^{(n)}) \geq -\sum_{x \in E_n}w_x\log w_x = \log Z_n - \frac{1}{Z_n}\sum_{x\in E_n}e^{S_n\phi(x)}S_n\phi(x).
\end{equation}
Also $\int S_n\phi\,d\mu_n = \frac{1}{Z_n}\sum_x e^{S_n\phi(x)}\cdot\frac{1}{n}\sum_{j=0}^{n-1}\phi(T^jx)\cdot n = \frac{n}{Z_n}\sum_x e^{S_n\phi(x)}\int\phi\,d\delta_x^{(n)}$. Combining and dividing by $n$:
\begin{equation}
\frac{1}{n}H_{\mu_n}(\alpha^{(n)}) + \int\phi\,d\mu_n \geq \frac{1}{n}\log Z_n \geq \frac{1}{n}\log Q_n(\phi,\epsilon) - \frac{\log 2}{n}.
\end{equation}
Using the upper semi-continuity of entropy (Appendix~\ref{app:entropy}) and continuity of $\mu \mapsto \int\phi\,d\mu$, passing to the limit $n_k \to \infty$: $h_\mu(T) + \int\phi\,d\mu \geq \limsup_n\frac{1}{n}\log Q_n(\phi,\epsilon)$. Taking $\epsilon \to 0$: $\sup_\mu\{h_\mu(T)+\int\phi\,d\mu\} \geq P(\phi)$.
\end{proof}

\begin{theorem}[Pressure as Convex Conjugate]\label{thm:pressure_conjugate}
Define the entropy functional $S: \calM_T(X) \to [-\infty, 0]$ by $S(\mu) = -h_\mu(T)$. Then:
\begin{equation}\label{eq:pressure_as_conjugate_thm}
P(\phi) = S^*(\phi) = \sup_{\mu \in \calM_T(X)} \left\{ \int \phi \, d\mu + h_\mu(T) \right\},
\end{equation}
and the entropy is recovered as the conjugate of the pressure:
\begin{equation}\label{eq:entropy_as_conjugate_thm}
S(\mu) = P^*(\mu) = \sup_{\phi \in C(X)} \left\{ \int \phi \, d\mu - P(\phi) \right\}.
\end{equation}
\end{theorem}

\begin{proof}
The first equality is the variational principle (Theorem~\ref{thm:variational_lower}).

For the second equality, we verify $S(\mu) = P^*(\mu)$ in two steps.

\textbf{Step 1} ($P^*(\mu) \leq S(\mu)$): For any $\phi \in C(X)$, the variational upper bound gives $h_\mu(T) + \int\phi\,d\mu \leq P(\phi)$, so $\int\phi\,d\mu - P(\phi) \leq -h_\mu(T) = S(\mu)$. Taking the supremum over $\phi$: $P^*(\mu) \leq S(\mu)$.

\textbf{Step 2} ($P^*(\mu) \geq S(\mu)$): We must show $\sup_\phi\{\int\phi\,d\mu - P(\phi)\} \geq -h_\mu(T)$. For any $\psi \in C(X)$ and $t \in \R$, the normalization $P(t\psi) \leq t\sup_X\psi + h_{\mathrm{top}}(T)$ (Proposition~\ref{prop:pressure_basic}(v)) gives $\int t\psi\,d\mu - P(t\psi) \geq t\int\psi\,d\mu - t\sup_X\psi - h_{\mathrm{top}}(T)$. Taking $\psi = 0$: $P^*(\mu) \geq 0 - P(0) = -h_{\mathrm{top}}(T)$. For a sharper bound, we use the Fenchel-Moreau theorem. The functional $-S(\mu) = h_\mu(T)$ is upper semi-continuous and concave (affine on ergodic decompositions). Its concave conjugate is $(- S)_*(\phi) = \inf_\mu\{\int\phi\,d\mu + h_\mu(T)\}$, and by the Fenchel-Moreau theorem for concave u.s.c.\ functions on compact convex sets (see Phelps \cite{Phelps1993}, Corollary~5.9), the biconjugate recovers the original: $h_\mu(T) = \inf_\phi\{P(\phi) - \int\phi\,d\mu\}$, which rearranges to $\sup_\phi\{\int\phi\,d\mu - P(\phi)\} = -h_\mu(T) = S(\mu)$.
\end{proof}

The duality between pressure and entropy has a clean geometric reading.
The pressure $P(\phi_0 + t\psi)$, viewed as a function of the real
parameter $t$, is convex (Theorem~\ref{thm:pressure_convexity}). At any
value of $t$ where this function is differentiable, the slope of the
unique tangent line equals $\int\psi\,d\mu_{\phi_0+t\psi}$, the
expectation of $\psi$ under the unique equilibrium state
(Theorem~\ref{thm:differentiability}). Convexity forces the entire
curve to lie above each tangent line. At a point of
non-differentiability, multiple supporting lines with different slopes
correspond to distinct equilibrium states coexisting at the same
potential: a first-order phase transition
(Theorem~\ref{thm:phase_transition_equiv}).

The negative entropy $S(\mu) = -h_\mu(T)$ is the convex conjugate of
$P$ restricted to the invariant measures
(Theorem~\ref{thm:pressure_conjugate}). As a conjugate of a convex function, $S$ is concave and attains its maximum at $-h_{\mathrm{top}}(T)$, where $h_{\mathrm{top}}(T)$ is the topological entropy. When the measure of maximal entropy is unique, we denote it by $\mu_{\mathrm{mme}}$; it corresponds to the potential $\phi = 0$. The biconjugate recovery $P = S^*$ closes the
duality: every property of the pressure is encoded in the entropy and
vice versa. Figure~\ref{fig:legendre_duality} illustrates this
correspondence.

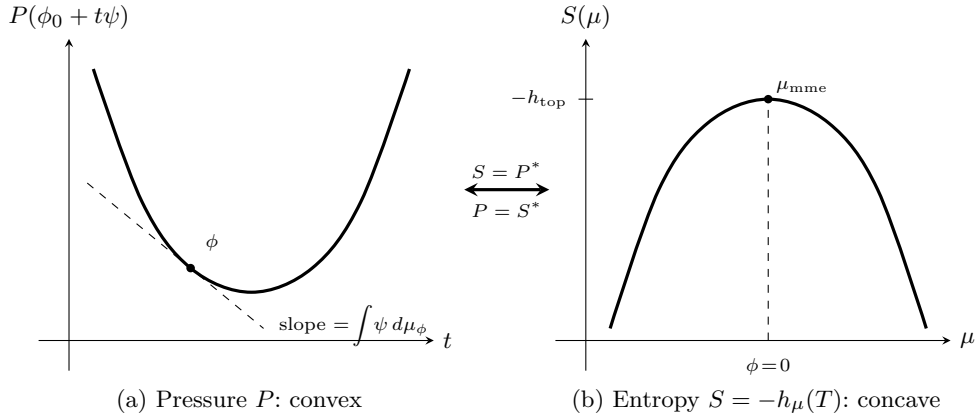
\begin{figure}[ht]
\centering
\begin{tikzpicture}[>=stealth, font=\small, scale=0.80]

\begin{scope}[xshift=0cm]
  \draw[->] (-0.5,0) -- (6.0,0) node[right] {$t$};
  \draw[->] (0,-0.5) -- (0,5.0) node[above] {$P(\phi_0 + t\psi)$};
  
  \draw[very thick] plot[smooth, tension=0.6] coordinates 
    {(0.4,4.5) (1.2,2.3) (2.0,1.2) (3.0,0.8) (4.0,1.2) (4.8,2.3) (5.6,4.5)};
  
  \draw[thin, dashed] (0.3,2.61) -- (3.2,0.20);
  
  \fill (2.0,1.2) circle (2pt);
  \node[above right=2pt, font=\footnotesize] at (2.0,1.25) {$\phi$};
  
  \node[right, font=\footnotesize] at (3.3,0.25) 
    {slope $= \!\int\!\psi\,d\mu_\phi$};
  
  
  \node[font=\small] at (2.8,-1.0) {(a) Pressure $P$: convex};
\end{scope}

\draw[<->, very thick] (6.5,2.5) -- (7.9,2.5);
\node[above=1pt, font=\footnotesize] at (7.2,2.5) {$S = P^*$};
\node[below=1pt, font=\footnotesize] at (7.2,2.5) {$P = S^*$};

\begin{scope}[xshift=8.5cm]
  \draw[->] (-0.5,0) -- (6.0,0) node[right] {$\mu$};
  \draw[->] (0,-0.5) -- (0,5.0) node[above] {$S(\mu)$};
  
  \draw[very thick] plot[smooth, tension=0.6] coordinates 
    {(0.4,0.2) (1.2,2.5) (2.0,3.6) (3.0,4.0) (4.0,3.6) (4.8,2.5) (5.6,0.2)};
  
  \fill (3.0,4.0) circle (2pt);
  
  \draw[thin, dashed] (3.0,0) -- (3.0,4.0);
  
  \node[right=4pt, font=\footnotesize] at (2.8,4.2) {$\mu_{\mathrm{mme}}$};
  
  \draw[thin] (-0.12,4.0) -- (0.12,4.0);
  \node[left, font=\footnotesize] at (-0.15,4.0) {$-h_{\mathrm{top}}$};
  
  \node[below=3pt, font=\footnotesize] at (3.0,0) {$\phi\!=\!0$};
  
  \node[font=\small] at (2.8,-1.0) {(b) Entropy $S = -h_\mu(T)$: concave};
\end{scope}

\end{tikzpicture}
\caption{The Legendre-Fenchel duality between pressure and entropy (Theorem~\ref{thm:pressure_conjugate}). (a)~The pressure $P(\phi_0 + t\psi)$ is convex; at a differentiable point, the unique tangent line has slope $\int\psi\,d\mu_\phi$, corresponding to the unique equilibrium state $\mu_\phi$. The curve lies entirely above the tangent, as required by convexity. (b)~The negative entropy $S(\mu) = -h_\mu(T)$ is concave and attains its maximum $-h_{\mathrm{top}}(T)$ at the measure of maximal entropy $\mu_{\mathrm{mme}}$ (corresponding to $\phi = 0$). The biconjugate recovery $S = P^*$, $P = S^*$ establishes the complete duality.}
\label{fig:legendre_duality}
\end{figure}


\subsection{Subdifferential Characterization and Structure}\label{sec:equilibrium_tangent}

With the duality $P = S^*$ established, the subdifferential calculus of Lemma~\ref{thm:subdiff_properties} immediately characterizes equilibrium states as elements of $\partial P(\phi)$. This subsection develops that characterization and its structural consequences: convexity, compactness, and the ergodic decomposition of the set of equilibrium states.

\begin{theorem}[Equilibrium States as Subdifferential]\label{thm:equilibrium_subdiff}
Let $(X, T)$ be a topological dynamical system with $X$ compact and $T$ continuous. For $\phi \in C(X)$, the following conditions on $\mu \in \calM_T(X)$ are equivalent:
\begin{enumerate}
\item[(i)] $\mu$ is an equilibrium state for $\phi$: $h_\mu(T) + \int \phi \, d\mu = P(\phi)$.
\item[(ii)] $\mu \in \partial P(\phi)$: for all $\psi \in C(X)$, $P(\psi) \geq P(\phi) + \int (\psi - \phi) \, d\mu$.
\item[(iii)] The pair $(\phi, \mu)$ satisfies the Fenchel-Young equality: $P(\phi) + S(\mu) = \int \phi \, d\mu$.
\item[(iv)] $\mu$ achieves the supremum in the variational formula.
\end{enumerate}
\end{theorem}

\begin{proof}
$(i) \Leftrightarrow (iii)$: The Fenchel-Young equality $P(\phi) + S(\mu) = \langle \phi, \mu \rangle$ becomes $P(\phi) - h_\mu(T) = \int \phi \, d\mu$, which is $h_\mu(T) + \int \phi \, d\mu = P(\phi)$.

$(i) \Leftrightarrow (iv)$: This is the definition of equilibrium state combined with the variational principle.

$(i) \Rightarrow (ii)$: Suppose $h_\mu(T) + \int \phi \, d\mu = P(\phi)$. For any $\psi \in C(X)$, Theorem~\ref{thm:variational_upper} gives $h_\mu(T) + \int \psi \, d\mu \leq P(\psi)$. Subtracting the equilibrium condition: $\int (\psi - \phi) \, d\mu \leq P(\psi) - P(\phi)$, which is the subdifferential inequality.

$(ii) \Rightarrow (i)$: Assume $\mu \in \partial P(\phi)$: $P(\psi) \geq P(\phi) + \int(\psi-\phi)\,d\mu$ for all $\psi \in C(X)$. By Theorem~\ref{thm:subdiff_properties}(iii), this is equivalent to the Fenchel-Young equality $P(\phi) + P^*(\mu) = \int\phi\,d\mu$. Since $P^*(\mu) = S(\mu) = -h_\mu(T)$ (Theorem~\ref{thm:pressure_conjugate}), this gives $P(\phi) - h_\mu(T) = \int\phi\,d\mu$, i.e., $h_\mu(T) + \int\phi\,d\mu = P(\phi)$, which is condition (i).
\end{proof}

\begin{corollary}[Convexity and Compactness]\label{cor:equilibrium_convex}
For any $\phi \in C(X)$, the set $\partial P(\phi)$ of equilibrium states is convex, weak-$*$ closed, weak-$*$ compact, and non-empty.
\end{corollary}

\begin{proof}
Convexity and weak-$*$ closedness follow from Theorem~\ref{thm:subdiff_properties}(i). Weak-$*$ compactness follows from the boundedness of $\partial P(\phi)$ (Theorem~\ref{thm:subdiff_properties}(ii), since $P$ is Lipschitz hence continuous) combined with weak-$*$ closedness. Non-emptiness: the function $\mu \mapsto h_\mu(T) + \int\phi\,d\mu$ is upper semi-continuous on the weak-$*$ compact set $\calM_T(X)$ (the entropy is u.s.c.\ by Theorem~\ref{thm:entropy_usc} for expansive systems, and in general by Theorem~\ref{thm:subdiff_properties}(ii) which guarantees $\partial P(\phi) \neq \emptyset$ for continuous convex $P$).
\end{proof}

\begin{corollary}[Ergodic Decomposition of Equilibrium States]\label{cor:equilibrium_ergodic}
The extreme points of $\partial P(\phi)$ are precisely the ergodic equilibrium states. Every equilibrium state is a convex combination (or integral) of ergodic equilibrium states.
\end{corollary}

\begin{proof}
We show $\partial P(\phi)$ is a face of $\calM_T(X)$. Suppose $\mu = \lambda\mu_1 + (1-\lambda)\mu_2$ with $\mu \in \partial P(\phi)$, $\mu_1,\mu_2 \in \calM_T(X)$, and $\lambda \in (0,1)$. Then:
\begin{align}
P(\phi) &= h_\mu(T) + \int\phi\,d\mu \\
&= \lambda\!\left(h_{\mu_1}(T)+\int\phi\,d\mu_1\right) + (1-\lambda)\!\left(h_{\mu_2}(T)+\int\phi\,d\mu_2\right),
\end{align}
using the affinity of $\mu \mapsto \int\phi\,d\mu$ and the affinity of entropy on ergodic decompositions (Theorem~\ref{thm:ergodic_decomposition}). Since $h_{\mu_i}(T)+\int\phi\,d\mu_i \leq P(\phi)$ for $i=1,2$ (variational upper bound), and their convex combination equals $P(\phi)$, both must equal $P(\phi)$. Thus $\mu_1,\mu_2 \in \partial P(\phi)$, confirming the face property.

Since $\partial P(\phi)$ is a face of the Choquet simplex $\calM_T(X)$, its extreme points are extreme points of $\calM_T(X)$, which are the ergodic measures (Proposition~\ref{prop:ergodic_extreme}). The integral representation follows from Theorem~\ref{thm:ergodic_decomposition} applied to $\partial P(\phi)$.
\end{proof}

\subsection{Differentiability of Pressure}

Differentiability of the pressure at $\phi$ is equivalent to uniqueness of the equilibrium state. When the equilibrium state is unique, the derivative of $P$ at $\phi$ is the linear functional $\psi \mapsto \int\psi\,d\mu_\phi$. This connection between convex analysis and ergodic theory is the mechanism behind both the generic uniqueness result and the phase transition characterization.

\begin{theorem}[Differentiability and Uniqueness]\label{thm:differentiability}
The following are equivalent:
\begin{enumerate}
\item[(i)] $P$ is G\^{a}teaux differentiable at $\phi$.
\item[(ii)] $\partial P(\phi)$ is a singleton.
\item[(iii)] There exists a unique equilibrium state for $\phi$.
\end{enumerate}
When these hold, $DP(\phi)(\psi) = \int_X \psi \, d\mu_\phi$.
\end{theorem}

\begin{proof}
$(i) \Leftrightarrow (ii)$: Theorem~\ref{thm:subdiff_properties}(iv). $(ii) \Leftrightarrow (iii)$: Theorem~\ref{thm:equilibrium_subdiff}. For the derivative formula: when $\partial P(\phi) = \{\mu_\phi\}$, the directional derivative is $P'(\phi;\psi) = \sup_{x^* \in \partial P(\phi)}\langle\psi,x^*\rangle = \int\psi\,d\mu_\phi$.
\end{proof}

\begin{corollary}[Generic Uniqueness]\label{cor:generic_uniqueness}
The set of potentials $\phi \in C(X)$ with a unique equilibrium state is a dense $G_\delta$ subset of $C(X)$.
\end{corollary}

\begin{proof}
By the Mazur theorem (see Phelps \cite{Phelps1993}, Theorem~1.20), a continuous convex function on a separable Banach space is G\^{a}teaux differentiable on a dense $G_\delta$ set. The pressure $P: C(X) \to \R$ is continuous and convex (Proposition~\ref{prop:pressure_basic} and Theorem~\ref{thm:pressure_convexity}), and $C(X)$ is separable (since $X$ is compact metric). By Theorem~\ref{thm:differentiability}, G\^{a}teaux differentiability at $\phi$ is equivalent to uniqueness of the equilibrium state.
\end{proof}

The relationship between differentiability, uniqueness, and phase
transitions is visible in the graph of the pressure along a
one-parameter family $\phi_0 + t\psi$. At a value of $t$ where the
graph is smooth, convexity guarantees a unique tangent line whose slope
equals $\int\psi\,d\mu_{\phi_0+t\psi}$; the equilibrium state is
unique (Theorem~\ref{thm:differentiability}). At a value of $t$ where
the graph has a corner, multiple supporting lines with different slopes
fit beneath the curve; each slope corresponds to a distinct equilibrium
state, and the potential exhibits a first-order phase transition
(Theorem~\ref{thm:phase_transition_equiv}). The jump in slope across
the corner equals the difference
$\int\psi\,d\mu_1 - \int\psi\,d\mu_2$ between the two coexisting
phases (Example~\ref{ex:two_phase}).
Figure~\ref{fig:pressure_tangent} illustrates this correspondence.

\begin{figure}[ht]
\centering
\hspace{-0.8cm}%
\begin{tikzpicture}[>=stealth, font=\small, scale=0.82]

  \draw[->] (-0.3,0) -- (10.2,0) node[right] {$t$};
  \draw[->] (0,-0.3) -- (0,6.0) node[above] {$P(\phi_0 + t\psi)$};


  \draw[very thick, domain=0.3:5, samples=60] 
    plot (\x, {1 - 0.3*(\x-5) + 0.08*(\x-5)*(\x-5)});
  \draw[very thick, domain=5:9.5, samples=60] 
    plot (\x, {1 + 0.6*(\x-5) + 0.08*(\x-5)*(\x-5)});

  \draw[thin, gray, dashed] 
    plot[domain=0.5:4.5] (\x, {4.18 - 0.78*\x});
  \fill (2, 2.62) circle (2pt);

  \draw[thin, ->] (0.6,4.8) -- (1.85,2.85);
  \node[above, font=\footnotesize] at (0.6,4.8) {$\phi_1$};

  \fill (5, 1.0) circle (3pt);

  \draw[thick, dashed, domain=1.5:7.5] 
    plot (\x, {2.5 - 0.3*\x});
  \draw[thick, dashed, domain=3.5:9.3] 
    plot (\x, {-2.0 + 0.6*\x});

  \draw[thin, ->] (5.0,-1.0) -- (5.0,0.85);
  \node[below, font=\footnotesize, align=center] at (5.0,-1.1) 
    {$\phi_c$: corner\\(phase transition)};

  \draw[thin, ->] (9,0.9) -- (7.5,0.3);
  \node[right, font=\footnotesize] at (8.5,1.0) 
    {slope $= \!\int\!\psi\,d\mu_1$};

  \draw[thin, ->] (9.8,4.8) -- (9.15,3.55);
  \node[above, font=\footnotesize] at (11,4.6) 
    {slope $= \!\int\!\psi\,d\mu_2$};

  \draw[thin, gray, dashed] 
    plot[domain=5.5:9.5] (\x, {-5.12 + 1.08*\x});
  \fill (8, 3.52) circle (2pt);

  \draw[thin, ->] (9.5,2.5) -- (8.15,3.45);
  \node[right, font=\footnotesize] at (9.5,2.5) {$\phi_2$};

\end{tikzpicture}
\caption{Pressure $P(\phi_0 + t\psi)$ along a one-parameter family 
(Theorems~\ref{thm:equilibrium_subdiff},~\ref{thm:differentiability}, 
and~\ref{thm:phase_transition_equiv}). At the smooth points $\phi_1$ 
and $\phi_2$, the unique tangent line (thin gray) has slope 
$\int\psi\,d\mu_\phi$, corresponding to a unique equilibrium state. 
At the corner $\phi_c$, two supporting lines (thick dashed) with 
different slopes pass through the same point, corresponding to two 
coexisting equilibrium states 
$\mu_1, \mu_2 \in \partial P(\phi_c)$: a first-order phase 
transition.}
\label{fig:pressure_tangent}
\end{figure}
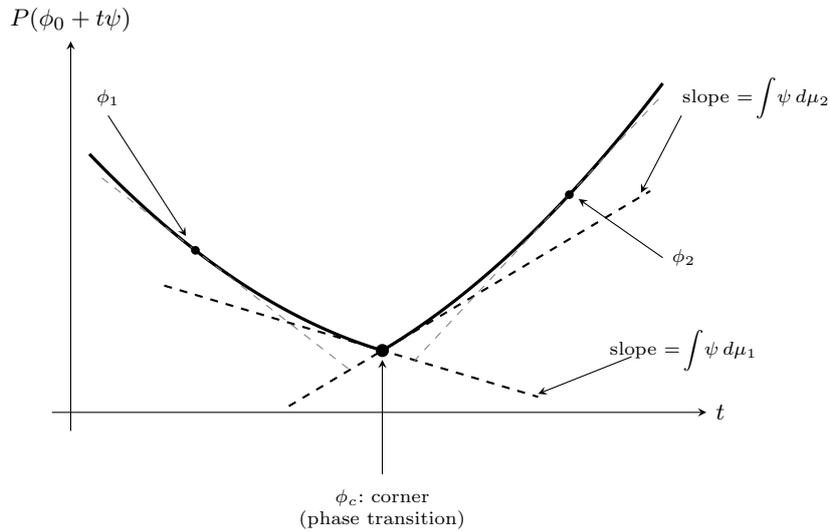

\subsection{Fr\'{e}chet Differentiability and the Variance Formula}

\begin{theorem}[Fr\'{e}chet Differentiability]\label{thm:frechet}
Let $(X, T)$ be an expansive dynamical system with specification. If $\phi \in C^\alpha(X)$ has a unique equilibrium state $\mu_\phi$ with exponential decay of correlations, then $P: C^\alpha(X) \to \R$ is Fr\'{e}chet differentiable at $\phi$ in the $\|\cdot\|_\alpha$ norm, with $DP(\phi)(\psi) = \int \psi \, d\mu_\phi$. On $C(X)$ with the $\|\cdot\|_\infty$ norm, the pressure is G\^{a}teaux differentiable at $\phi$ with the same derivative.
\end{theorem}

\begin{proof}
\textbf{Step 1 (Analytic perturbation).} Through the Markov partition (guaranteed by specification), we work in the symbolic setting. The Ruelle transfer operator $\calL_{\phi+t\psi}$ for $\psi \in C(X)$ and $t$ in a complex neighborhood of $0$ forms an analytic family of bounded operators on the H\"{o}lder space $C^\alpha(X)$. By the spectral theory (Part~I \cite{Thiam2026a}, Theorem~2.2), $\calL_\phi$ has a simple dominant eigenvalue $\lambda_0 = e^{P(\phi)}$ with spectral gap $\gamma < 1$. By Kato's analytic perturbation theory \cite{Kato1980}, the dominant eigenvalue $\lambda(t)$ of $\calL_{\phi+t\psi}$ is analytic in $t$ for $|t| < \epsilon$, where $\epsilon$ depends on the spectral gap and $\|\psi\|_\alpha$.

\textbf{Step 2 (Pressure is analytic).} Since $P(\phi+t\psi) = \log\lambda(t)$ and $\lambda(t) > 0$ is analytic, the map $t \mapsto P(\phi+t\psi)$ is real-analytic for each $\psi \in C^\alpha(X)$. By the Taylor expansion at $t = 0$:
\begin{equation}
P(\phi+t\psi) = P(\phi) + t\int\psi\,d\mu_\phi + \frac{t^2}{2}D^2P(\phi)(\psi,\psi) + O(t^3\|\psi\|_\alpha^3).
\end{equation}

\textbf{Step 3 (Fr\'{e}chet estimate).} Setting $t = 1$ with $\psi$ replaced by $\psi$ of small norm: $|P(\phi+\psi) - P(\phi) - \int\psi\,d\mu_\phi| \leq C\|\psi\|_\alpha^2$, where the bound on $|D^2P(\phi)(\psi,\psi)|$ follows from the resolvent estimate $\|(\lambda_0 - \calL_0)^{-1}Q\| \leq (1-\gamma)^{-1}\lambda_0^{-1}$, giving $|D^2P(\phi)(\psi,\psi)| \leq C_0\|\psi\|_\alpha^2$ with $C_0$ depending on $\phi$ and the spectral gap. This is the Fr\'{e}chet condition on $C^\alpha(X)$.

\textbf{Step 4 (G\^{a}teaux differentiability on $C(X)$).} For general $\psi \in C(X)$ and $t \in \R$ small, the map $t \mapsto P(\phi + t\psi)$ is convex. The one-sided derivative satisfies $P'(\phi;\psi) = \sup_{\mu \in \partial P(\phi)}\int\psi\,d\mu = \int\psi\,d\mu_\phi$ since $\partial P(\phi) = \{\mu_\phi\}$ (Theorem~\ref{thm:differentiability}). As $P'(\phi;\psi) = -P'(\phi;-\psi) = \int\psi\,d\mu_\phi$ for all $\psi \in C(X)$, the pressure is G\^{a}teaux differentiable at $\phi$ on $C(X)$. The Fr\'{e}chet condition $|P(\phi+\psi)-P(\phi)-\int\psi\,d\mu_\phi| = o(\|\psi\|_\infty)$ does not hold uniformly for all $\psi \in C(X)$, since continuous functions with small supremum norm can have arbitrarily poor moduli of continuity; the Fr\'{e}chet estimate is restricted to $C^\alpha(X)$.
\end{proof}

\begin{theorem}[Second Derivative of Pressure]\label{thm:second_derivative}
Suppose $(X, T)$ has specification and $\phi \in C^\alpha(X)$ has a unique equilibrium state $\mu_\phi$ with exponential mixing. For $\psi \in C^\alpha(X)$,
\begin{equation}\label{eq:second_derivative}
D^2 P(\phi)(\psi, \psi) = \lim_{n \to \infty} \frac{1}{n} \Var_{\mu_\phi}(S_n \psi) = \xi^2_\psi \geq 0,
\end{equation}
with $\xi^2_\psi = 0$ if and only if $\psi$ is cohomologous to a constant.
\end{theorem}

\begin{proof}
Write $\calL_t = \calL_{\phi+t\psi}$ with eigendata $\lambda_t, h_t, \nu_t$. Without loss of generality assume $\int\psi\,d\mu_\phi = 0$ (so $\lambda'(0)=0$). Differentiating $\calL_t h_t = \lambda_t h_t$ at $t=0$:
\begin{equation}
\calL_0 h'(0) + \psi\calL_0 h_0 = \lambda'(0)h_0 + \lambda_0 h'(0).
\end{equation}
Since $\lambda'(0) = 0$, we get $(\lambda_0 - \calL_0)h'(0) = \psi\calL_0 h_0$. Let $Q = I - h_0\nu_0$ be the complementary spectral projection. Then $h'(0) = (\lambda_0-\calL_0)^{-1}Q(\psi h_0)$, where the resolvent $(\lambda_0-\calL_0)^{-1}Q = \sum_{n=0}^\infty\lambda_0^{-(n+1)}\calL_0^n Q$ converges by the spectral gap.

Differentiating $\calL_t h_t = \lambda_t h_t$ twice at $t=0$ and integrating against $\nu_0$:
\begin{equation}
\lambda_0\!\int h''(0)\,d\nu_0 + 2\!\int\psi\calL_0 h'(0)\,d\nu_0 + \int\psi^2 h_0\,d\nu_0 = \lambda''(0) + \lambda_0\!\int h''(0)\,d\nu_0.
\end{equation}
The $\int h''(0)\,d\nu_0$ terms cancel, giving $\lambda''(0) = 2\int\psi\calL_0 h'(0)\,d\nu_0 + \int\psi^2\,d\mu_0$. Substituting the resolvent expression:
\begin{equation}
\int\psi\calL_0 h'(0)\,d\nu_0 = \sum_{n=1}^\infty\int\psi\cdot(\psi\circ T^n)h_0\,d\nu_0 = \sum_{n=1}^\infty\mathrm{Cov}_{\mu_\phi}(\psi,\psi\circ T^n),
\end{equation}
where the series converges absolutely by exponential mixing: $|\mathrm{Cov}_{\mu_\phi}(\psi,\psi\circ T^n)| \leq C\gamma^n\|\psi\|_\alpha^2$. Therefore:
\begin{equation}
P''(\phi;\psi) = \frac{\lambda''(0)}{\lambda_0} = \mathrm{Var}_{\mu_\phi}(\psi) + 2\sum_{n=1}^\infty\mathrm{Cov}_{\mu_\phi}(\psi,\psi\circ T^n).
\end{equation}
Expanding $\mathrm{Var}_{\mu_\phi}(S_n\psi) = \sum_{j,k=0}^{n-1}\mathrm{Cov}_{\mu_\phi}(\psi\circ T^j,\psi\circ T^k) = n\mathrm{Var}_{\mu_\phi}(\psi) + 2\sum_{m=1}^{n-1}(n-m)\mathrm{Cov}_{\mu_\phi}(\psi,\psi\circ T^m)$, dividing by $n$, and using dominated convergence: $\lim_{n\to\infty}n^{-1}\mathrm{Var}_{\mu_\phi}(S_n\psi) = \mathrm{Var}_{\mu_\phi}(\psi) + 2\sum_{m=1}^\infty\mathrm{Cov}_{\mu_\phi}(\psi,\psi\circ T^m) = P''(\phi;\psi)$.

Non-negativity follows since $P$ is convex. For the zero characterization: $\xi^2_\psi = 0$ implies $\mathrm{Var}_{\mu_\phi}(S_n\psi) = o(n)$, hence $n^{-1/2}S_n\psi \to 0$ in $L^2(\mu_\phi)$. By the theory of stationary processes with spectral gap, this holds if and only if $\psi = u \circ T - u + c$ for some $u \in L^2(\mu_\phi)$ and $c = \int\psi\,d\mu_\phi = 0$ (see Parry and Pollicott \cite{ParryPollicott1990}, Chapter~4).
\end{proof}

\section{Phase Transitions and Non-Uniqueness}\label{sec:phase_transitions}

A first-order phase transition occurs when the pressure $P$ is not differentiable at a potential $\phi$, which by Theorem~\ref{thm:differentiability} is equivalent to the coexistence of multiple equilibrium states. This section analyzes the geometry of phase coexistence: the set of equilibrium states at a transition point forms a simplex (Theorem~\ref{thm:coexistence_structure}), every ergodic equilibrium state is an exposed point (Proposition~\ref{thm:exposed_equilibrium}), and the directional derivative selects which phase is preferred under perturbation. We also collect criteria for uniqueness and phase transitions, with examples illustrating both possibilities.

\subsection{Characterization and Geometry of Phase Coexistence}

\begin{definition}[First-Order Phase Transition]\label{def:phase_transition}
A potential $\phi \in C(X)$ exhibits a first-order phase transition if $|\partial P(\phi)| > 1$.
\end{definition}

\begin{theorem}[Equivalent Characterizations of Phase Transitions]\label{thm:phase_transition_equiv}
The following conditions are equivalent:
\begin{enumerate}
\item[(i)] $\phi$ exhibits a first-order phase transition.
\item[(ii)] $P$ is not G\^{a}teaux differentiable at $\phi$.
\item[(iii)] There exist directions $\psi$ with $P'(\phi;\psi) + P'(\phi;-\psi) > 0$.
\item[(iv)] $\partial P(\phi)$ is not a singleton.
\end{enumerate}
\end{theorem}

\begin{proof}
$(i) \Leftrightarrow (iv)$: definition combined with Theorem~\ref{thm:equilibrium_subdiff}. $(ii) \Leftrightarrow (iv)$: Theorem~\ref{thm:differentiability}. $(ii) \Leftrightarrow (iii)$: G\^{a}teaux differentiability at $\phi$ means $P'(\phi;\psi) = -P'(\phi;-\psi)$ for all $\psi$ (i.e., the directional derivative is linear). Negating: $P$ is not differentiable if and only if there exists $\psi$ with $P'(\phi;\psi) > -P'(\phi;-\psi)$, i.e., $P'(\phi;\psi)+P'(\phi;-\psi) > 0$. (For convex $P$, we always have $P'(\phi;\psi) \geq -P'(\phi;-\psi)$.)
\end{proof}

\begin{theorem}[Structure of Coexisting Phases]\label{thm:coexistence_structure}
Suppose $\phi$ exhibits a first-order phase transition with ergodic equilibrium states $\mu_1,\ldots,\mu_k$ ($k \geq 2$). Then:
\begin{enumerate}
\item[(i)] $\partial P(\phi) = \mathrm{conv}\{\mu_1,\ldots,\mu_k\}$ is a $(k-1)$-simplex.
\item[(ii)] The one-sided directional derivative satisfies $P'(\phi;\psi) = \max_{1\leq i\leq k}\int\psi\,d\mu_i$.
\item[(iii)] (Selection principle) If $\int\psi\,d\mu_1 > \int\psi\,d\mu_j$ for all $j \geq 2$, then $\lim_{t\to 0^+}\partial P(\phi+t\psi) = \{\mu_1\}$.
\end{enumerate}
\end{theorem}

\begin{proof}
Part (i): By Corollary~\ref{cor:equilibrium_ergodic}, $\partial P(\phi)$ is a face of $\calM_T(X)$ whose extreme points are the ergodic equilibrium states $\mu_1,\ldots,\mu_k$. Since $\calM_T(X)$ is a Choquet simplex, every face is a simplex, so $\partial P(\phi) = \mathrm{conv}\{\mu_1,\ldots,\mu_k\}$.

Part (ii): For convex $P$, $P'(\phi;\psi) = \sup\{\int\psi\,d\mu : \mu \in \partial P(\phi)\}$ (this is a standard result in convex analysis; see Phelps \cite{Phelps1993}, Proposition~1.8). On the simplex $\partial P(\phi) = \mathrm{conv}\{\mu_1,\ldots,\mu_k\}$, the linear functional $\mu \mapsto \int\psi\,d\mu$ achieves its supremum at a vertex, giving $P'(\phi;\psi) = \max_i\int\psi\,d\mu_i$.

Part (iii): For $t > 0$, any equilibrium state $\nu$ for $\phi+t\psi$ satisfies $h_\nu(T)+\int(\phi+t\psi)\,d\nu = P(\phi+t\psi)$. By the variational upper bound applied to each $\mu_i$: $P(\phi+t\psi) \geq h_{\mu_i}(T)+\int\phi\,d\mu_i + t\int\psi\,d\mu_i = P(\phi)+t\int\psi\,d\mu_i$. Since $\int\psi\,d\mu_1 > \int\psi\,d\mu_j$ for $j \geq 2$, the measure $\mu_1$ gives a strictly larger value than any other $\mu_j$ for the functional $h_\nu(T)+\int(\phi+t\psi)\,d\nu$. As $t \to 0^+$, any weak-$*$ limit of equilibrium states for $\phi+t\psi$ must lie in $\partial P(\phi)$ (by upper semi-continuity of the subdifferential correspondence) and must maximize $\int\psi\,d\mu$ over $\partial P(\phi)$. The unique maximizer is $\mu_1$, so $\partial P(\phi+t\psi) \to \{\mu_1\}$.
\end{proof}

\begin{example}[Two-Phase Coexistence]\label{ex:two_phase}
With $k = 2$ ergodic equilibrium states $\mu_1, \mu_2$ and $\int\psi\,d\mu_1 > \int\psi\,d\mu_2$: $P'(\phi;\psi) = \int\psi\,d\mu_1$ and $P'(\phi;-\psi) = -\int\psi\,d\mu_2$. The ``jump'' is $P'(\phi;\psi)+P'(\phi;-\psi) = \int\psi\,d\mu_1-\int\psi\,d\mu_2 > 0$.
\end{example}

\begin{definition}[Exposed Equilibrium States]\label{def:exposed_equilibrium}
An equilibrium state $\mu$ for $\phi$ is exposed if there exists $\psi \in C(X)$ such that $\mu$ is the unique equilibrium state for $\phi + t\psi$ for all sufficiently small $t > 0$.
\end{definition}

\begin{proposition}[Exposed Points of the Subdifferential]\label{thm:exposed_equilibrium}
Every ergodic equilibrium state is an exposed point of $\partial P(\phi)$.
\end{proposition}

\begin{proof}
Let $\mu$ be an ergodic equilibrium state for $\phi$, and let $\mu_1,\ldots,\mu_k$ be all the ergodic equilibrium states (with $\mu = \mu_1$, say). Since the $\mu_i$ are distinct ergodic measures, they are mutually singular. By the Hahn decomposition, for each $j \geq 2$ there exists a Borel set $A_j$ with $\mu_1(A_j) = 1$ and $\mu_j(A_j) = 0$. Let $A = \bigcap_{j\geq 2}A_j$; then $\mu_1(A) = 1$ and $\mu_j(A) = 0$ for $j \geq 2$. By Urysohn's lemma, there exists $\psi \in C(X)$ with $0 \leq \psi \leq 1$, $\psi = 1$ on a compact $K \subset A$ with $\mu_1(K) > 1/2$, and $\psi = 0$ outside an open neighborhood of $K$ with $\mu_j$-measure less than $1/4$ for each $j \geq 2$. Then $\int\psi\,d\mu_1 > 1/2$ and $\int\psi\,d\mu_j < 1/4$ for $j \geq 2$. By Theorem~\ref{thm:coexistence_structure}(iii), $\lim_{t\to 0^+}\partial P(\phi+t\psi) = \{\mu_1\}$. Thus $\mu_1$ is exposed by $\psi$.
\end{proof}

\subsection{Criteria and Examples}

\begin{theorem}[Uniqueness Criteria]\label{thm:uniqueness_criteria}
The following conditions each imply that $\phi$ has a unique equilibrium state:
\begin{enumerate}
\item[(i)] (Specification and Regularity) $(X, T)$ has specification and $\phi$ is H\"{o}lder continuous.
\item[(ii)] (Spectral Gap) $(X, T)$ is expansive with specification, $\phi$ has summable variations, and $\calL_\phi$ has a spectral gap.
\item[(iii)] (Small Perturbation) $(X, T)$ is expansive with specification, and $\phi_0 \in C(X)$ has summable variations with unique equilibrium state $\mu_{\phi_0}$. Then there exists $\epsilon > 0$ (depending on the spectral gap of $\calL_{\phi_0}$) such that any $\psi$ with $\|\psi\|_\alpha < \epsilon$ yields a unique equilibrium state for $\phi_0 + \psi$.
\end{enumerate}
\end{theorem}

\begin{proof}
Part (i): This is the Bowen-Ruelle theorem. Under specification, the Markov partition lifts $\phi$ to the symbolic space where it lies in $\calH_\alpha$. The Ruelle-Perron-Frobenius theorem (Part~I \cite{Thiam2026a}, Theorem~2.2) provides a spectral gap, and uniqueness follows from Part~I \cite{Thiam2026a}, Main Theorem~2.1(iv). See also Bowen \cite{Bowen1974}.

Part (ii): When $\calL_\phi$ has a spectral gap, the leading eigenvalue $\lambda = e^{P(\phi)}$ is simple and isolated. The eigenmeasure $\nu$ and eigenfunction $h$ are unique (Part~I \cite{Thiam2026a}, Theorem~2.2(a)--(b)), so the Gibbs measure $\mu = h\nu$ is the unique equilibrium state by Part~I \cite{Thiam2026a}, Main Theorem~2.1.

Part (iii): By Part~(ii), the transfer operator $\calL_{\phi_0}$ on $\calH_\alpha$ has a spectral gap: the dominant eigenvalue $\lambda_0 = e^{P(\phi_0)}$ is simple and isolated, with the rest of the spectrum in $\{|z| \leq \gamma_0\lambda_0\}$ for some $\gamma_0 < 1$ (Part~I \cite{Thiam2026a}, Theorem~2.2(c)). Set $\phi = \phi_0 + \psi$. Since $\calL_\phi g = \calL_{\phi_0}(e^\psi g)$ (where $(e^\psi g)(y) = e^{\psi(y)}g(y)$), we have $\calL_\phi - \calL_{\phi_0} = \calL_{\phi_0} \circ M_{e^\psi - 1}$, where $M_f$ denotes pointwise multiplication by $f$. The Banach algebra property of $\calH_\alpha$ gives $\|M_{e^\psi - 1}\|_{\calH_\alpha \to \calH_\alpha} \leq \|e^\psi - 1\|_\alpha$, and for $\|\psi\|_\alpha \leq 1$ we have $\|e^\psi - 1\|_\alpha \leq C'\|\psi\|_\alpha$ (using $\|e^\psi - 1\|_\infty \leq e^{\|\psi\|_\infty}\|\psi\|_\infty$ and $|e^\psi - 1|_\alpha \leq e^{\|\psi\|_\infty}|\psi|_\alpha$). Therefore
\begin{equation}
\|\calL_\phi - \calL_{\phi_0}\|_{\calH_\alpha \to \calH_\alpha} \leq \|\calL_{\phi_0}\|_{\calH_\alpha \to \calH_\alpha} \cdot C'\|\psi\|_\alpha \leq C_0\|\psi\|_\alpha
\end{equation}
for $\|\psi\|_\alpha \leq 1$, where $C_0$ depends on $\phi_0$. By the stability of simple isolated eigenvalues under bounded perturbations (Kato \cite{Kato1980}, Chapter~IV, \S3.5), there exists $\epsilon > 0$ such that for $\|\psi\|_\alpha < \epsilon$, the operator $\calL_\phi$ has a simple dominant eigenvalue $\lambda_\phi$ near $\lambda_0$, separated from the rest of the spectrum by a gap of width at least $(1-\gamma_0)\lambda_0/2 > 0$. The corresponding eigenfunction $h_\phi > 0$ and eigenmeasure $\nu_\phi$ are the unique eigendata (Part~I \cite{Thiam2026a}, Theorem~2.2(a)--(b) applied to $\phi$), and the Gibbs measure $\mu_\phi = h_\phi\nu_\phi$ is the unique equilibrium state (Part~I \cite{Thiam2026a}, Main Theorem~2.1). The threshold $\epsilon$ is determined by $\epsilon < (1-\gamma_0)\lambda_0/(2C_0)$, which depends on the spectral gap $\gamma_0$ and the operator norm $C_0$. See also Bowen \cite{Bowen1974}, Theorem~1.2.
\end{proof}

\begin{theorem}[Phase Transition Criteria]\label{thm:phase_criteria}
The following conditions produce phase transitions:
\begin{enumerate}
\item[(i)] If the system has multiple measures of maximal entropy, then $\phi = 0$ exhibits a phase transition.
\item[(ii)] For certain systems, potentials $t\psi$ with $|t|$ sufficiently large can have multiple equilibrium states.
\item[(iii)] In the zero-temperature limit $\beta \to \infty$, the equilibrium states for $\beta\phi$ concentrate on the ground states, which may not be unique.
\end{enumerate}
\end{theorem}

\begin{proof}
Part (i): If $\mu_1, \mu_2$ are distinct measures of maximal entropy, then $h_{\mu_1}(T) = h_{\mu_2}(T) = h_{\mathrm{top}}(T) = P(0)$. Both satisfy the equilibrium condition $h_{\mu_i}(T) + \int 0\,d\mu_i = P(0)$, so $\mu_1,\mu_2 \in \partial P(0)$ and $|\partial P(0)| > 1$.

Part (ii): Consider a system $(X,T)$ with specification and two disjoint closed invariant subsystems $X_1, X_2$ supporting ergodic measures $\mu_1, \mu_2$ with $h_{\mu_1}(T) = h_{\mu_2}(T) = h_{\mathrm{top}}(T)$. Let $\psi \in C(X)$ satisfy $\int\psi\,d\mu_1 > 0 > \int\psi\,d\mu_2$. For $\phi_t = t\psi$, the free energy of $\mu_i$ is $h_{\mu_i}(T) + t\int\psi\,d\mu_i$. At $t = 0$ both are equilibrium states. For $t > 0$, $\mu_1$ is preferred; for $t < 0$, $\mu_2$ is preferred. By the upper semi-continuity of the subdifferential correspondence $t \mapsto \partial P(t\psi)$ and the fact that $\partial P(0)$ contains both $\mu_1$ and $\mu_2$, there exists a transition at $t = 0$. See Hofbauer \cite{Hofbauer1977} for explicit constructions on subshifts.

Part (iii): For any equilibrium state $\mu_\beta$ of $\beta\phi$: $h_{\mu_\beta}(T) + \beta\int\phi\,d\mu_\beta = P(\beta\phi)$. Dividing by $\beta$: $\beta^{-1}h_{\mu_\beta}(T) + \int\phi\,d\mu_\beta = \beta^{-1}P(\beta\phi)$. As $\beta \to \infty$: $\beta^{-1}h_{\mu_\beta}(T) \to 0$ (since $0 \leq h_{\mu_\beta}(T) \leq h_{\mathrm{top}}(T)$) and $\beta^{-1}P(\beta\phi) \to \max_\mu\int\phi\,d\mu$ (since $P(\beta\phi)/\beta \leq \sup_X\phi + \beta^{-1}h_{\mathrm{top}}(T) \to \sup_X\phi$ and $P(\beta\phi)/\beta \geq \int\phi\,d\mu + \beta^{-1}h_\mu(T) \to \int\phi\,d\mu$ for any $\mu$). Thus any weak-$*$ limit of $\mu_\beta$ maximizes $\int\phi\,d\mu$. When the set of maximizing measures contains more than one ergodic measure $\nu_1, \nu_2$ with $\int\phi\,d\nu_1 = \int\phi\,d\nu_2 = \max_\mu\int\phi\,d\mu$, the subdifferential $\partial P(\beta\phi)$ must contain both $\nu_1$ and $\nu_2$ for some $\beta$ (by the upper hemicontinuity of $\beta \mapsto \partial P(\beta\phi)$ and the fact that distinct accumulation points of $\mu_\beta$ as $\beta \to \infty$ must coexist at a finite transition point). For the complete analysis of the zero-temperature limit, see Chazottes and Hochman \cite{ChazottesHochman2010}.
\end{proof}

\begin{example}[Manneville-Pomeau Maps]\label{ex:manneville_pomeau}
The map $T(x) = x + x^{1+\alpha} \pmod 1$ for $\alpha > 0$ has an indifferent fixed point at $0$. The geometric potential $\phi_t(x) = -t \log |T'(x)|$ exhibits a phase transition at $t = t_c$: for $t < t_c$ the equilibrium state is absolutely continuous, and for $t > t_c$ it is the Dirac mass at the fixed point.
\end{example}

\begin{example}[Hofbauer's Example]\label{ex:hofbauer}
Hofbauer \cite{Hofbauer1977} constructed a continuous potential on a subshift of finite type with two distinct ergodic equilibrium states, balanced between two invariant subsystems.
\end{example}

\section{The Universal Variational Principle}\label{sec:universal}

The classical, subadditive, and relative variational principles share a common structure: each equates a convex functional on potentials with a supremum over invariant measures. This section proves that any functional satisfying four natural axioms (convexity, lower semi-continuity, coercivity, and cocycle invariance) admits a unique Legendre-Fenchel representation (Theorem~\ref{thm:universal_variational}), thereby unifying these three variational principles in a single theorem.

\subsection{Abstract Setting}

\begin{definition}[Admissible Pressure Functional]\label{def:admissible_pressure}
A functional $\Phi: \calE \to \R \cup \{+\infty\}$ on a Banach space $\calE \supset C(X)$ is admissible if it is convex, lower semi-continuous, coercive ($\Phi(\phi) \geq a\|\phi\| - b$), normalized ($\Phi(c) = c$ for constants), and monotone.
\end{definition}

\begin{theorem}[Universal Variational Principle]\label{thm:universal_variational}
Let $\Phi$ be an admissible pressure functional. Then there exists a unique $\Sigma: \calM_T(X) \to [-\infty, 0]$, concave and upper semi-continuous, such that
\begin{equation}
\Phi(\phi) = \sup_{\mu \in \calM_T(X)} \{\langle \phi, \mu \rangle + \Sigma(\mu)\},
\end{equation}
with $\Sigma = \Phi^*|_{\calM_T(X)}$.
\end{theorem}

\begin{proof}
\textbf{Step 1 (Construction of $\Sigma$).} Define $\Sigma(\mu) = \inf_{\phi \in \calE}\{\Phi(\phi) - \int\phi\,d\mu\} = -\Phi^*(\mu)$ for $\mu \in \calM_T(X)$. By definition of the Legendre-Fenchel transform, $\Phi(\phi) \geq \int\phi\,d\mu + \Sigma(\mu)$ for all $\phi, \mu$, so $\Phi(\phi) \geq \sup_\mu\{\int\phi\,d\mu + \Sigma(\mu)\}$.

\textbf{Step 2 (Equality).} By the Fenchel-Moreau theorem (Theorem~\ref{thm:LF_properties}(iv)), since $\Phi$ is proper, convex, and l.s.c., $\Phi = \Phi^{**}$. Thus $\Phi(\phi) = \sup_{\mu \in \calE^*}\{\int\phi\,d\mu - \Phi^*(\mu)\}$. We show the supremum can be restricted to $\calM_T(X)$ by proving $\Phi^*(\mu) = +\infty$ for all $\mu \notin \calM_T(X)$.

\emph{Case 1:} $\mu(X) \neq 1$. By normalization, $\Phi(c) = c$ for every constant $c \in \R$. Thus $\Phi^*(\mu) \geq \sup_{c \in \R}\{c\mu(X) - c\} = \sup_{c \in \R}\{c(\mu(X)-1)\} = +\infty$ since $\mu(X) \neq 1$.

\emph{Case 2:} $\mu(X) = 1$ but $\mu$ is not a positive measure. Then there exists a non-positive $\phi_0 \in C(X)$ with $\phi_0 \leq 0$ and $\int\phi_0\,d\mu > 0$ (choose any non-negative $f$ with $\int f\,d\mu < 0$, which exists since $\mu$ is not positive, and set $\phi_0 = -f$). By monotonicity, $t\phi_0 \leq 0$ implies $\Phi(t\phi_0) \leq \Phi(0)$ for $t > 0$. Thus $\Phi^*(\mu) \geq \sup_{t > 0}\{t\int\phi_0\,d\mu - \Phi(0)\} = +\infty$ since $\int\phi_0\,d\mu > 0$.

\emph{Case 3:} $\mu \in \calP(X)$ but $\mu$ is not $T$-invariant. Then there exists $\phi_0 \in C(X)$ with $\int\phi_0\circ T\,d\mu \neq \int\phi_0\,d\mu$. Set $\psi = \phi_0 - \phi_0\circ T$, so $\int\psi\,d\mu \neq 0$. By cocycle invariance, $\Phi(t\psi) = \Phi(t\phi_0 - t\phi_0\circ T) = \Phi(0) = h_{\mathrm{top}}(T)$ for all $t$ (since $t\psi$ is a coboundary $t\phi_0 - t\phi_0\circ T$, and $\Phi$ is cocycle-invariant). Thus $\Phi^*(\mu) \geq \sup_{t \in \R}\{t\int\psi\,d\mu - h_{\mathrm{top}}(T)\} = +\infty$ since $\int\psi\,d\mu \neq 0$.

\textbf{Step 3 (Properties of $\Sigma$).} Concavity: $-\Sigma(\mu) = \Phi^*(\mu)$ is convex (Theorem~\ref{thm:LF_properties}(i)), so $\Sigma$ is concave. Upper semi-continuity: $-\Sigma$ is weak-$*$ l.s.c.\ (Theorem~\ref{thm:LF_properties}(i)), so $\Sigma$ is weak-$*$ u.s.c. Non-positivity: by the normalization hypothesis $\Phi(c) = c$ for any constant $c \in \R$, we have $\Sigma(\mu) = \inf_\phi\{\Phi(\phi)-\int\phi\,d\mu\} \leq \Phi(c) - \int c\,d\mu = c - c\cdot\mu(X) = c - c = 0$ for any constant $c$, since $\mu$ is a probability measure. Uniqueness: if $\tilde\Sigma$ also represents $\Phi$, then $\tilde\Sigma(\mu) = \inf_\phi\{\Phi(\phi)-\int\phi\,d\mu\} = \Sigma(\mu)$.
\end{proof}

\subsection{Subadditive and Relative Variational Principles}

\begin{definition}[Subadditive Sequence]\label{def:subadditive_sequence}
A sequence $\Psi = (\psi_n)_{n \geq 1}$ of continuous functions is subadditive if $\psi_{m+n} \leq \psi_m + \psi_n \circ T^m$ for all $m, n \geq 1$.
\end{definition}

\begin{definition}[Subadditive Pressure]\label{def:subadditive_pressure}
$P(\Psi) = \lim_{\epsilon \to 0} \limsup_{n \to \infty} \frac{1}{n} \log \sup\{ \sum_{x \in E} e^{\psi_n(x)} : E \text{ is } (n, \epsilon)\text{-separated}\}$.
\end{definition}

\begin{theorem}[Subadditive Variational Principle]\label{thm:subadditive_VP}
For a subadditive sequence $\Psi$ of continuous functions,
\begin{equation}\label{eq:subadditive_variational}
P(\Psi) = \sup_{\mu \in \calM_T(X)} \left\{ h_\mu(T) + \calF_*(\Psi, \mu) \right\},
\end{equation}
where $\calF_*(\Psi, \mu) = \lim_{n \to \infty} \frac{1}{n} \int \psi_n \, d\mu$ exists by Kingman's subadditive ergodic theorem.
\end{theorem}

\begin{proof}
\textbf{Upper bound.} For any $\mu \in \calM_T(X)$ and finite partition $\calU$ with Lebesgue number $\delta > 0$, the Gibbs inequality (Lemma~\ref{lem:gibbs_inequality_abstract}) applied to the join $\calU^{(n)} = \bigvee_{k=0}^{n-1}T^{-k}\calU$ gives:
\begin{equation}
H_\mu(\calU^{(n)}) + \int\psi_n\,d\mu \leq \log\sum_{U\in\calU^{(n)}}\exp(\sup_U\psi_n).
\end{equation}
The right side is bounded by $\log Q_n^{\mathrm{sub}}(\Psi,\delta) + C$ where $Q_n^{\mathrm{sub}}$ counts the partition sum for the subadditive sequence. Dividing by $n$ and taking $n\to\infty$: $h_\mu(T,\calU) + \calF_*(\Psi,\mu) \leq P(\Psi) + C'(\delta)$. Taking $\delta \to 0$: $h_\mu(T) + \calF_*(\Psi,\mu) \leq P(\Psi)$.

\textbf{Lower bound.} Assume $(X,T)$ is expansive (needed for upper semi-continuity of entropy; for non-expansive systems, the result holds with the variational entropy of Cao et~al. \cite{CaoFengHuang2008}). For $\varepsilon > 0$, let $E_n$ be a maximal $(n,\varepsilon)$-separated set. Define the auxiliary measure $\nu_n = Z_n^{-1}\sum_{x\in E_n}e^{\psi_n(x)}\delta_x$ where $Z_n = \sum_{x\in E_n}e^{\psi_n(x)}$, and set $\mu_n = n^{-1}\sum_{j=0}^{n-1}T_*^j\nu_n$. Let $\mu$ be a weak-$*$ limit of a subsequence $\mu_{n_k}$. Then $\mu \in \calM_T(X)$.

By Kingman's subadditive ergodic theorem, $\calF_*(\Psi,\mu) = \lim_{n\to\infty} n^{-1}\int\psi_n\,d\mu$ exists. The entropy estimate follows by the same grouping argument as Theorem~\ref{thm:specification_equilibrium}: each $x \in E_n$ lies in a distinct atom $A_x$ of a partition $\calP_\varepsilon^{(n)}$ with $\diam(\calP_\varepsilon) < \varepsilon$, so $H_{\nu_n}(\calP_\varepsilon^{(n)}) = -\sum_x w_x\log w_x$ where $w_x = e^{\psi_n(x)}/Z_n$. The exact identity $H_{\nu_n}(\calP_\varepsilon^{(n)}) + \sum_x w_x\psi_n(x) = \log Z_n$, combined with the grouping bound $\frac{1}{p}H_{\mu_n}(\calP_\varepsilon^{(p)}) \geq \frac{1}{n}H_{\nu_n}(\calP_\varepsilon^{(n)}) - \frac{p\log|\calP_\varepsilon|}{n}$ and the integral identity $\frac{1}{n}\sum_x w_x\psi_n(x) = \frac{1}{n}\int\psi_n\,d\nu_n$, gives:
\begin{equation}
\frac{1}{n}H_{\mu_n}(\calP_\varepsilon^{(n)}) + \frac{1}{n}\sum_{x}w_x\psi_n(x) \geq \frac{1}{n}\log Z_n.
\end{equation}
The second term on the left converges to $\calF_*(\Psi,\mu)$ (by construction of $\mu_n$). Passing to the limit: $h_\mu(T) + \calF_*(\Psi,\mu) \geq \liminf_{n\to\infty} n^{-1}\log Z_n \geq P(\Psi) - C(\varepsilon)$. Since $\varepsilon$ is arbitrary, $\sup_\mu\{h_\mu(T)+\calF_*(\Psi,\mu)\} \geq P(\Psi)$.
\end{proof}

\begin{definition}[Factor Map, Conditional Entropy, Relative Pressure]\label{def:relative}
A factor map $\pi: (X, T) \to (Y, S)$ is a continuous surjection with $\pi \circ T = S \circ \pi$. The conditional entropy is $h_\mu(T | \pi) = h_\mu(T) - h_{\pi_* \mu}(S)$. The relative pressure is $P(\phi | \pi) = \sup_{\nu \in \calM_S(Y)} P_\nu(\phi)$.
\end{definition}

\begin{theorem}[Relative Variational Principle]\label{thm:relative_VP}
For a factor map $\pi: (X, T) \to (Y, S)$ and $\phi \in C(X)$,
\begin{equation}\label{eq:relative_variational}
P(\phi | \pi) = \sup_{\mu \in \calM_T(X)} \left\{ h_\mu(T | \pi) + \int \phi \, d\mu \right\}.
\end{equation}
\end{theorem}

\begin{proof}
This follows from the Ledrappier-Walters formula \cite{LedrappierWalters1977}. The entropy decomposes as $h_\mu(T) = h_\mu(T|\pi) + h_{\pi_*\mu}(S)$ (the Abramov-Rokhlin formula; see Walters \cite{Walters1982}, Theorem~8.3). Applying the classical variational principle to each fiber:

\textbf{Upper bound.} For $\mu \in \calM_T(X)$, write $\nu = \pi_*\mu$. Then $h_\mu(T|\pi) + \int\phi\,d\mu = h_\mu(T) - h_\nu(S) + \int\phi\,d\mu \leq P(\phi) - h_\nu(S)$. Taking the supremum over $\mu$ with fixed $\pi_*\mu = \nu$: $\sup_{\mu:\pi_*\mu=\nu}\{h_\mu(T|\pi)+\int\phi\,d\mu\} \leq P_\nu(\phi)$, the fiber pressure. Then $\sup_\mu\{h_\mu(T|\pi)+\int\phi\,d\mu\} \leq \sup_\nu P_\nu(\phi) = P(\phi|\pi)$.

\textbf{Lower bound.} For each $\nu \in \calM_S(Y)$ and $\epsilon > 0$, the fiber variational principle gives $\mu_\nu \in \calM_T(X)$ with $\pi_*\mu_\nu = \nu$ and $h_{\mu_\nu}(T|\pi)+\int\phi\,d\mu_\nu \geq P_\nu(\phi)-\epsilon$. Taking $\nu$ achieving $\sup_\nu P_\nu(\phi) - \epsilon$: $\sup_\mu\{h_\mu(T|\pi)+\int\phi\,d\mu\} \geq P(\phi|\pi) - 2\epsilon$. Since $\epsilon$ is arbitrary, equality holds.
\end{proof}

\section{Equilibrium States for Systems with Specification}\label{sec:specification}

The specification property, introduced by Bowen \cite{Bowen1974}, guarantees that any finite collection of orbit segments can be approximated by a single periodic orbit. For systems with specification, Theorem~\ref{thm:specification_equilibrium} establishes existence of equilibrium states for all continuous potentials, uniqueness for potentials with summable variations, and exponential mixing of the unique equilibrium state. The proofs connect the abstract variational theory of the preceding sections to the concrete spectral theory of Part~I \cite{Thiam2026a}.

\begin{definition}[Specification Property]\label{def:specification}
$(X, T)$ has the specification property if for every $\epsilon > 0$, there exists $M(\epsilon) \in \N$ such that any finite collection of orbit segments can be $\epsilon$-shadowed by a single periodic orbit with gaps of length at most $M(\epsilon)$.
\end{definition}

\begin{theorem}[Equilibrium States for Specification Systems]\label{thm:specification_equilibrium}
Let $(X, T)$ have specification. For any $\phi \in C(X)$:
\begin{enumerate}
\item[(i)] There exists at least one equilibrium state.
\item[(ii)] If $\phi$ has summable variations, the equilibrium state is unique.
\item[(iii)] The unique equilibrium state is mixing with exponential decay of correlations.
\end{enumerate}
\end{theorem}

\begin{proof}
\textbf{Part (i).} For each $n$ and $\epsilon > 0$, let $E_n$ be a maximal $(n,\epsilon)$-separated set with $Z_n = \sum_{x \in E_n} e^{S_n\phi(x)} \geq \frac{1}{2}Q_n(\phi,\epsilon)$. Define the auxiliary measure $\nu_n = Z_n^{-1}\sum_{x \in E_n} e^{S_n\phi(x)}\delta_x$ and the time-averaged measure $\mu_n = \frac{1}{n}\sum_{j=0}^{n-1}T_*^j\nu_n = Z_n^{-1}\sum_{x \in E_n} e^{S_n\phi(x)} \delta_x^{(n)}$, where $\delta_x^{(n)} = \frac{1}{n}\sum_{j=0}^{n-1}\delta_{T^j x}$. By compactness, some subsequence $\mu_{n_k} \to \mu$ weak-$*$. The limit $\mu$ is $T$-invariant (the invariance error is $O(1/n)$).

To show $\mu$ achieves the pressure, we use the approach of Walters \cite{Walters1982}, Theorem~8.6. Let $\alpha$ be a partition with $\diam(\alpha) < \epsilon/2$. Since $E_n$ is $(n,\epsilon)$-separated and atoms of $\alpha^{(n)} = \bigvee_{k=0}^{n-1}T^{-k}\alpha$ have $d_n$-diameter less than $\epsilon$, each $x \in E_n$ lies in a distinct atom $A_x$ of $\alpha^{(n)}$. Therefore $\nu_n(A_x) = w_x = e^{S_n\phi(x)}/Z_n$ and $\nu_n(A) = 0$ for atoms $A$ not containing any point of $E_n$. The partition entropy of $\nu_n$ is exactly:
\begin{equation}
H_{\nu_n}(\alpha^{(n)}) = -\sum_{x \in E_n}w_x\log w_x = \log Z_n - \sum_x w_x S_n\phi(x).
\end{equation}
Since $\mu_n = \frac{1}{n}\sum_{j=0}^{n-1}T_*^j\nu_n$, we have $\int\phi\,d\mu_n = \frac{1}{n}\int S_n\phi\,d\nu_n = \frac{1}{n}\sum_x w_x S_n\phi(x)$ by the telescoping identity $\sum_{j=0}^{n-1}\int\phi\circ T^j\,d\nu_n = \int S_n\phi\,d\nu_n$. For any fixed $p \geq 1$, the concavity of partition entropy under convex combinations gives:
\begin{equation}
H_{\mu_n}(\alpha^{(p)}) \geq \frac{1}{n}\sum_{j=0}^{n-1}H_{T_*^j\nu_n}(\alpha^{(p)}) = \frac{1}{n}\sum_{j=0}^{n-1}H_{\nu_n}(T^{-j}\alpha^{(p)}).
\end{equation}
To bound this sum from below, partition $\{0,1,\ldots,n-1\}$ into $p$ groups $G_l = \{l, l+p, l+2p,\ldots\}$ for $l = 0,\ldots,p-1$. For each group $G_l$, the partition $\alpha^{(n)} = \bigvee_{k=0}^{n-1}T^{-k}\alpha$ is coarser than $\bigvee_{j \in G_l}T^{-j}\alpha^{(p)}$ (since each $T^{-j}\alpha^{(p)} = \bigvee_{k=j}^{j+p-1}T^{-k}\alpha$ contributes $p$ factors, and the groups cover all indices $0,\ldots,n+p-2$). By subadditivity of entropy:
\begin{equation}
H_{\nu_n}(\alpha^{(n)}) \leq \sum_{j \in G_l}H_{\nu_n}(T^{-j}\alpha^{(p)}) + p\log|\alpha|,
\end{equation}
where the error $p\log|\alpha|$ accounts for the boundary indices beyond $n-1$. Summing over $l = 0,\ldots,p-1$ and using the fact that $\bigcup_{l=0}^{p-1}G_l = \{0,\ldots,n-1\}$:
\begin{equation}
pH_{\nu_n}(\alpha^{(n)}) \leq \sum_{j=0}^{n-1}H_{\nu_n}(T^{-j}\alpha^{(p)}) + p^2\log|\alpha|.
\end{equation}
Combining with the concavity bound and the exact identity $H_{\nu_n}(\alpha^{(n)}) = \log Z_n - \sum_x w_x S_n\phi(x)$:
\begin{equation}
\frac{1}{p}H_{\mu_n}(\alpha^{(p)}) + \int\phi\,d\mu_n \geq \frac{1}{n}H_{\nu_n}(\alpha^{(n)}) - \frac{p\log|\alpha|}{n} + \frac{1}{n}\sum_x w_x S_n\phi(x) = \frac{1}{n}\log Z_n - \frac{p\log|\alpha|}{n}.
\end{equation}
Now $Z_n = \sum_{x \in E_n}e^{S_n\phi(x)} \geq \frac{1}{2}Q_n(\phi,\epsilon)$ by choice of $E_n$. The specification property guarantees $\limsup_{n\to\infty}\frac{1}{n}\log Q_n(\phi,\epsilon) = P(\phi) + o_\epsilon(1)$ where $o_\epsilon(1) \to 0$ as $\epsilon \to 0$. Passing to the subsequential limit $\mu_{n_k} \to \mu$: by the upper semi-continuity of $\nu \mapsto H_\nu(\alpha^{(p)})$ (for fixed $p$ and a partition with $\mu$-negligible boundaries), $\frac{1}{p}H_\mu(\alpha^{(p)}) + \int\phi\,d\mu \geq P(\phi) - o_\epsilon(1)$. Taking $p \to \infty$: $h_\mu(T,\alpha) + \int\phi\,d\mu \geq P(\phi) - o_\epsilon(1)$. Since $\alpha$ was chosen with $\diam(\alpha) < \epsilon/2$, taking $\epsilon \to 0$ (and refining $\alpha$ to generate): $h_\mu(T)+\int\phi\,d\mu \geq P(\phi)$. The reverse inequality is the variational upper bound (Theorem~\ref{thm:variational_upper}), so $\mu$ is an equilibrium state.

\textbf{Parts (ii) and (iii).} Under summable variations, $\phi$ lifted to the symbolic space (via the Markov partition from specification) lies in $\calF_A$. By Part~I \cite{Thiam2026a}, Theorem~2.2, the transfer operator $\calL_\phi$ has a spectral gap. Uniqueness: the spectral gap implies the dominant eigenvalue $\lambda = e^{P(\phi)}$ is simple with unique eigenfunction $h > 0$ and eigenmeasure $\nu$ (Part~I \cite{Thiam2026a}, Theorem~2.2(a)--(b)). The Gibbs measure $\mu = h\nu$ is the unique measure satisfying the Jacobian condition $J_\mu\sigma = e^{P(\phi)-\phi}$ (Part~I \cite{Thiam2026a}, Main Theorem~2.1), hence the unique equilibrium state (Part~I \cite{Thiam2026a}, Main Theorem~2.1(iv)). Exponential mixing: by Part~I \cite{Thiam2026a}, Theorem~2.2(d), $\|\lambda^{-n}\calL_\phi^n g - \nu(g)h\|_\alpha \leq C\gamma^n\|g\|_\alpha$ with $\gamma < 1$, which translates to $|\mathrm{Cov}_\mu(f,g\circ T^n)| \leq C\gamma^n\|f\|_\alpha\|g\|_\alpha$ (Part~I \cite{Thiam2026a}, Theorem~2.6).
\end{proof}

\section{Extensions to Non-Compact Spaces}\label{sec:noncompact}

The variational principle of Section~\ref{sec:pressure_convex} requires $X$ to be compact. This section extends the theory to locally compact $\sigma$-compact spaces under a coercivity condition that prevents mass from escaping to infinity. The Gurevich pressure replaces the topological pressure, and Sarig's recurrence classification \cite{Sarig1999,Sarig2001,Sarig2003} governs when equilibrium states exist.

\begin{definition}[Coercive Potential]\label{def:coercive}
A potential $\phi: X \to \R \cup \{-\infty\}$ on a locally compact $\sigma$-compact space $X$ is coercive if $\phi(x) \leq -c \cdot d(x, K) - D$ for $x \notin K$, for some compact $K$ and $c, D > 0$.
\end{definition}

\begin{theorem}[Non-Compact Variational Principle]\label{thm:noncompact_VP}
For a coercive potential $\phi$ on a locally compact $\sigma$-compact $(X, T)$, the Gurevich pressure $P_G(\phi) = \limsup_{n \to \infty} n^{-1} \log \sum_{T^n x = x, x \in K} e^{S_n \phi(x)}$ satisfies $P_G(\phi) = \sup_{\mu \in \calM_T(X)} \{ h_\mu(T) + \int \phi \, d\mu \}$.
\end{theorem}

\begin{proof}
\textbf{Step 1 (Upper bound).} For $\mu \in \calM_T(X)$ with $h_\mu(T) < \infty$ and $\int\phi\,d\mu > -\infty$, the argument of Theorem~\ref{thm:variational_upper} extends: the Gibbs inequality and the Shannon-McMillan-Breiman theorem require only finite entropy, not compactness. Thus $h_\mu(T) + \int\phi\,d\mu \leq P_G(\phi)$.

\textbf{Step 2 (Lower bound via compact exhaustion).} Let $K_1 \subset K_2 \subset \cdots$ be an increasing exhaustion of $X$ by compact sets with $K_1 = K$. Define $X_n = \bigcap_{k=0}^{n-1}T^{-k}K_n$ and $P_n(\phi) = \lim_{m\to\infty}m^{-1}\log\sum_{T^mx=x, x\in K_n}e^{S_m\phi(x)}$. The compact variational principle (Theorem~\ref{thm:variational_lower}) applied to $K_n$ gives $P_n(\phi) = \sup_{\mu \in \calM_T(K_n)}\{h_\mu(T)+\int\phi\,d\mu\}$.

The coercivity of $\phi$ ensures that orbits spending significant time outside $K_n$ contribute negligibly to the partition sum: if $x \notin K_n$, then $\phi(x) \leq -c\cdot d(x,K)-D$, so the contribution of such orbits to $\sum e^{S_m\phi}$ is exponentially suppressed. Therefore $P_n(\phi) \nearrow P_G(\phi)$ as $n \to \infty$, and $\sup_n\sup_{\mu\in\calM_T(K_n)}\{h_\mu(T)+\int\phi\,d\mu\} = P_G(\phi)$.

For the full details, see Sarig \cite{Sarig1999}, Theorem~3.
\end{proof}

\begin{theorem}[Sarig's Theorem]\label{thm:sarig}
Let $(\Sigma_A, \sigma)$ be a topologically mixing countable Markov shift with BIP. For $\phi$ with summable variations and finite Gurevich pressure:
\begin{enumerate}
\item[(i)] Positive recurrence $\Rightarrow$ unique equilibrium probability measure.
\item[(ii)] Null recurrence $\Rightarrow$ unique conformal measure (up to scaling), no equilibrium probability measure.
\item[(iii)] Transience $\Rightarrow$ no conformal measure.
\end{enumerate}
\end{theorem}

\begin{proof}[Proof outline]
\textbf{Step 1.} BIP ensures $P_G(\phi)$ is independent of the reference state and $P_G(\phi) = \lim_n P(\phi_n)$ for finite-alphabet restrictions $\phi_n$. See Sarig \cite{Sarig1999}, Proposition~2.

\textbf{Step 2.} Positive recurrence provides $(\lambda,h,\nu)$ with $\calL_\phi h = \lambda h$, $\calL_\phi^*\nu = \lambda\nu$, $\lambda = e^{P_G(\phi)}$, $h > 0$, and $\int h\,d\nu < \infty$. The measure $\mu = (h/\int h\,d\nu)\nu$ is a probability measure satisfying $J_\mu\sigma = e^{P_G(\phi)-\phi}$ (Part~I \cite{Thiam2026a}, Main Theorem~2.1), hence an equilibrium state. Uniqueness follows from the spectral gap of $\calL_\phi$ under BIP and summable variations. See Sarig \cite{Sarig2001}, Theorem~2.

\textbf{Step 3.} In the null recurrent case, $\int h\,d\nu = \infty$, so no finite Gibbs measure exists, but the conformal measure $\nu$ (with $\calL_\phi^*\nu = \lambda\nu$) is unique up to scaling. In the transient case, the eigenequation has no solution. See Sarig \cite{Sarig2001}, Theorems~3--4, and Sarig \cite{Sarig2003}.
\end{proof}

\section{A Numerical Illustration: Pressure and Duality on the Golden Mean Shift}\label{sec:numerical}

We illustrate the convex-analytic theory with an explicit computation for the golden mean shift, making the abstract Legendre-Fenchel duality tangible with computable numbers.

\subsection{Setup}

The golden mean shift is the subshift of finite type $(\Sigma_A, \sigma)$ over $\{1, 2\}$ with transition matrix $A = \binom{1\ 1}{1\ 0}$ (the symbol $2$ cannot follow itself). The topological entropy is $h_{\mathrm{top}} = \log\varphi$ where $\varphi = (1+\sqrt{5})/2$ is the golden ratio. The mixing time is $M = 2$.

Define the one-parameter family of H\"{o}lder potentials
\begin{equation}
\phi_t(x) = t \cdot \mathbf{1}_{[x_0 = 1]}, \quad t \in \R,
\end{equation}
where $\mathbf{1}_{[x_0=1]}$ is the indicator of the cylinder $\{x : x_0 = 1\}$. This is a locally constant (hence H\"{o}lder with $|\phi_t|_\alpha = 0$) potential that assigns weight $t$ to the symbol $1$ and weight $0$ to the symbol $2$.

\subsection{Exact Computation of the Pressure}

The transfer operator for $\phi_t$ on the one-sided golden mean shift acts on functions of the first coordinate via the $2 \times 2$ matrix
\begin{equation}
A(t) = \begin{pmatrix} e^t & e^t \\ 1 & 0 \end{pmatrix}.
\end{equation}
The pressure is $P(\phi_t) = \log\lambda(t)$ where $\lambda(t)$ is the leading eigenvalue of $A(t)$. The characteristic polynomial $\lambda^2 - e^t\lambda - e^t = 0$ gives
\begin{equation}\label{eq:pressure_golden}
\lambda(t) = \frac{e^t + \sqrt{e^{2t} + 4e^t}}{2}, \quad P(\phi_t) = \log\lambda(t).
\end{equation}
At $t = 0$: $\lambda(0) = \varphi$, $P(0) = \log\varphi \approx 0.4812$ (the topological entropy). As $t \to +\infty$: $P(\phi_t) \sim t$ (the measure concentrates on the fixed point $\overline{1}$). As $t \to -\infty$: $P(\phi_t) \to 0$ (the measure concentrates on the orbit $\overline{21}$, which avoids the cylinder $[1]$ as much as possible given the constraint that $2$ cannot follow $2$).

The function $t \mapsto P(\phi_t)$ is convex and real-analytic on $\R$ (Main Theorem~\ref{thm:pressure_legendre_main}(ii) and Theorem~\ref{thm:frechet}), confirming the absence of phase transitions for this one-parameter family.

\subsection{Verification of the Legendre-Fenchel Duality}

Set $g = \mathbf{1}_{[x_0=1]}$, so $\phi_t = tg$. By Main Theorem~\ref{thm:subdifferential_main}, the unique equilibrium state $\mu_t$ for $\phi_t$ satisfies
\begin{equation}
P(\phi_t) = h_{\mu_t}(\sigma) + t\int g\,d\mu_t.
\end{equation}
The entropy of $\mu_t$ is $S(\mu_t) = -h_{\mu_t}(\sigma) = -(P(\phi_t) - t\bar{g}_t)$ where $\bar{g}_t = \int g\,d\mu_t = \mu_t([1])$. The Legendre-Fenchel duality $P = S^*$ (Main Theorem~\ref{thm:pressure_legendre_main}(i)) states
\begin{equation}
P(\phi_t) = \sup_\mu \left\{t\int g\,d\mu + h_\mu(\sigma)\right\} = \sup_{a \in [a_{\min}, a_{\max}]} \{ta - I(a)\}
\end{equation}
where $I(a) = \sup_t\{ta - P(\phi_t)\}$ is the Legendre transform of $P$ and $a = \int g\,d\mu$. The biconjugate recovery $S = P^*$ is verified: $I(a) = -h_{\mu_a}(\sigma)$ where $\mu_a$ is the unique measure with $\int g\,d\mu_a = a$.

\subsection{First Derivative: The Equilibrium State Mean}

By Main Theorem~\ref{thm:differentiability_main}(ii), the derivative of the pressure gives the equilibrium state mean:
\begin{equation}
P'(\phi_t; g) = \frac{d}{dt}P(\phi_t) = \int g\,d\mu_t = \mu_t([1]).
\end{equation}
Differentiating $P(\phi_t) = \log\lambda(t)$ using \eqref{eq:pressure_golden}:
\begin{equation}
P'(\phi_t; g) = \frac{\lambda'(t)}{\lambda(t)}.
\end{equation}
At $t = 0$: $\lambda'(0)/\lambda(0) = 1/\varphi \approx 0.6180$, confirming that $\mu_{\mathrm{mme}}([1]) = 1/\varphi$ (the frequency of symbol $1$ under the measure of maximal entropy). For general $t$, differentiating the characteristic equation $\lambda^2 = e^t\lambda + e^t$ implicitly gives $2\lambda\lambda' = e^t\lambda + e^t\lambda' + e^t$, so
\begin{equation}
P'(\phi_t; g) = \frac{e^t(\lambda(t) + 1)}{2\lambda(t)^2 - e^t\lambda(t)}.
\end{equation}
As $t$ increases from $-\infty$ to $+\infty$, $P'(\phi_t; g)$ increases monotonically from $1/2$ to $1$, reflecting the shift of the equilibrium state from concentrating on the orbit $\overline{21}$ (where symbol $1$ has frequency $1/2$) to concentrating on the fixed point $\overline{1}$ (frequency $1$).

\subsection{Second Derivative: The Asymptotic Variance}

By the variance formula (Theorem~\ref{thm:second_derivative}), the second derivative gives the asymptotic variance:
\begin{equation}
P''(\phi_t; g) = \frac{d^2}{dt^2}P(\phi_t) = \lim_{n \to \infty} \frac{1}{n}\mathrm{Var}_{\mu_t}(S_n g) = \sigma^2(g; \mu_t) \geq 0,
\end{equation}
with equality if and only if $g$ is cohomologous to a constant (Main Theorem~\ref{thm:differentiability_main}). At $t = 0$:
\begin{equation}
\sigma^2(g; \mu_{\mathrm{mme}}) = P''(0; g) = \frac{d^2}{dt^2}\log\lambda(t)\bigg|_{t=0} = \frac{1}{5\sqrt{5}} \approx 0.08944.
\end{equation}
The strict positivity $\sigma^2 > 0$ confirms that $g - \bar{g}$ is not a coboundary (consistent with the Liv\v{s}ic criterion, since $g$ takes values $\{0, 1\}$ and $\bar{g} = 1/\varphi$ is irrational). The strict convexity $P'' > 0$ confirms uniqueness of the equilibrium state for all $t$ (Main Theorem~\ref{thm:differentiability_main}(i)).

\subsection{Summary of Explicit Constants}

\begin{center}
\begin{tabular}{ll}
\textbf{Quantity} & \textbf{Value} \\[4pt]
Alphabet size $N$ & 2 \\
Mixing time $M$ & 2 \\
Topological entropy $P(0)$ & $\log\varphi \approx 0.4812$ \\
Leading eigenvalue $\lambda(0)$ & $\varphi = (1+\sqrt{5})/2 \approx 1.6180$ \\
Mean $P'(0; g) = \mu_{\mathrm{mme}}([1])$ & $1/\varphi \approx 0.6180$ \\
Variance $P''(0; g) = \sigma^2$ & $1/(5\sqrt{5}) \approx 0.08944$ \\
Range of $P'(\phi_t; g)$ as $t$ varies & $(1/2, 1)$ \\
Phase transitions & None ($P$ is real-analytic) \\
\end{tabular}
\end{center}

\noindent This example demonstrates that the Legendre-Fenchel duality $P = S^*$, the subdifferential characterization of equilibrium states, the first-derivative formula $P'(\phi; g) = \int g\,d\mu_\phi$, and the second-derivative variance formula $P''(\phi; g) = \sigma^2$ all produce computable numbers from the transition matrix and the potential, confirming the quantitative nature of the convex-analytic theory developed in this Part.

\section{Conclusion}\label{sec:conclusion}

This Part develops the convex-analytic structure of the thermodynamic
formalism for continuous maps on compact metric spaces. The pressure
functional $P: C(X) \to \R$ is the Legendre-Fenchel transform of the
negative entropy $S(\mu) = -h_\mu(T)$, and the biconjugate recovery
$S = P^*$ establishes a complete duality
(Theorem~\ref{thm:pressure_conjugate}). From this identification,
equilibrium states are elements of the subdifferential $\partial P(\phi)$
(Theorem~\ref{thm:equilibrium_subdiff}), uniqueness of the equilibrium
state is equivalent to G\^{a}teaux differentiability of $P$ at $\phi$
(Theorem~\ref{thm:differentiability}), and first-order phase transitions
correspond to non-differentiability
(Theorem~\ref{thm:phase_transition_equiv}). For systems with
specification and H\"{o}lder potentials, the pressure is Fr\'{e}chet
differentiable in the H\"{o}lder norm
(Theorem~\ref{thm:frechet}), and the second derivative equals the
asymptotic variance of the Birkhoff sums
(Theorem~\ref{thm:second_derivative}). The Universal Variational
Principle (Theorem~\ref{thm:universal_variational}) shows that the
classical, subadditive, and relative variational principles are
instances of a single duality between convex functionals satisfying
four axioms and concave entropy-like functionals on the space of
invariant measures. The theory extends to non-compact spaces under
coercivity conditions, where Sarig's recurrence classification
(Theorem~\ref{thm:sarig}) determines the existence and uniqueness of
equilibrium states for countable Markov shifts.

This convex-analytic structure complements the spectral theory of
Part~I \cite{Thiam2026a}. Part~I constructs Gibbs measures through the transfer operator
and spectral gap; this Part characterizes them
through the geometry of the pressure functional. Part~III \cite{Thiam2026c} connects
both approaches to smooth dynamics through quantitative Markov
partitions, providing the coding map that transfers the spectral and
variational results to Axiom~A diffeomorphisms.

\subsection*{Open Problems}

\begin{enumerate}
\item[] \textbf{Higher-order phase transitions.} The convex-analytic approach detects first-order phase transitions (non-differentiability of $P$). Can it detect second-order transitions (discontinuities in the second derivative)? For H\"{o}lder potentials on mixing SFTs, the pressure is analytic (Part~I \cite{Thiam2026a}, analyticity of pressure), so no phase transitions occur. Characterizing the regularity threshold at which transitions appear remains open.

\item[] \textbf{Infinite-dimensional subdifferentials.} For continuous (non-H\"{o}lder) potentials, $\partial P(\phi)$ can contain uncountably many ergodic measures. What is the geometric structure (dimension, topology) of $\partial P(\phi)$ in such cases?

\item[] \textbf{Non-compact pressure beyond coercivity.} Section~\ref{sec:noncompact} requires coercivity. The Legendre-Fenchel duality extends formally without coercivity, but the supremum in the variational principle may not be achieved. Can the formalism be extended to characterize when equilibrium states exist without the coercivity assumption?
\end{enumerate}

\appendix

\section{Entropy Estimates}\label{app:entropy}

This appendix provides the entropy estimates used in the main text. The key result is the quantitative upper semi-continuity of the entropy map $\mu \mapsto h_\mu(T)$ for expansive systems, which is needed for the compactness arguments in the variational principle.

\begin{proposition}[Quantitative Upper Semi-Continuity]\label{thm:entropy_usc}
Let $(X, T)$ be an expansive system with expansiveness constant $\delta > 0$. There exists $\omega: [0,\infty) \to [0,\infty)$ with $\omega(\epsilon) \to 0$ as $\epsilon \to 0$ such that $d_W(\mu,\nu) < \epsilon \implies h_\nu(T) \leq h_\mu(T) + \omega(\epsilon)$.
\end{proposition}

\begin{proof}
\textbf{Step 1.} Fix a partition $\alpha = \{A_1,\ldots,A_k\}$ with $\diam(A_i) < \delta$ and $\mu(\partial A_i) = 0$ for all $i$ (possible for all but countably many choices of boundary). By expansiveness, $h_\mu(T) = h_\mu(T,\alpha)$ for every $\mu$ (Walters \cite{Walters1982}, Theorem~8.2).

\textbf{Step 2.} For fixed $n$, the partition entropy $\mu \mapsto n^{-1}H_\mu(\alpha^{(n)})$ is continuous at measures $\mu$ for which all atoms of $\alpha^{(n)}$ have $\mu$-measure-zero boundaries. Since $h_\mu(T,\alpha) = \inf_n n^{-1}H_\mu(\alpha^{(n)})$ (decreasing limit), the entropy is upper semi-continuous as an infimum of continuous functions.

\textbf{Step 3 (Quantitative bound).} If $d_W(\mu,\nu) < \epsilon$, then for each atom $A$ of $\alpha^{(n)}$ with $\mu(\partial A) = 0$, we have $|\mu(A)-\nu(A)| \leq C_\alpha\epsilon$ where $C_\alpha$ depends on the minimum distance from $\partial A$ to its complement. Using the estimate $|H_\mu(\calP)-H_\nu(\calP)| \leq \delta_0\log|\calP| + H(\delta_0)$ where $\delta_0 = \max_A|\mu(A)-\nu(A)|$ and $H(\delta_0) = -\delta_0\log\delta_0-(1-\delta_0)\log(1-\delta_0)$ is the binary entropy (this follows from Fano's inequality):
\begin{equation}
|n^{-1}H_\mu(\alpha^{(n)}) - n^{-1}H_\nu(\alpha^{(n)})| \leq C_\alpha\epsilon\log k^n/n + n^{-1}H(C_\alpha\epsilon) = C_\alpha\epsilon\log k + o(1).
\end{equation}
Taking $n \to \infty$: $|h_\mu(T,\alpha)-h_\nu(T,\alpha)| \leq C_\alpha\epsilon\log k$. More carefully, $h_\nu(T) = h_\nu(T,\alpha) \leq \liminf_n n^{-1}H_\nu(\alpha^{(n)}) \leq \liminf_n(n^{-1}H_\mu(\alpha^{(n)}) + C_\alpha\epsilon\log k + o(1)) = h_\mu(T) + C_\alpha\epsilon\log k$. Setting $\omega(\epsilon) = C_\alpha\epsilon\log k$ gives the result.
\end{proof}

\section{Infinite-Dimensional Convex Analysis}\label{app:convex}

This appendix collects the results from infinite-dimensional convex analysis that are used in the main text. The Fenchel-Moreau theorem guarantees biconjugate recovery for lower semi-continuous convex functions, and the Moreau-Rockafellar theorem provides the sum rule for subdifferentials. Both results extend from finite to infinite dimensions under the constraint qualifications satisfied by the pressure functional.

\begin{proposition}[Fenchel-Moreau]\label{thm:fenchel_moreau}
Let $E$ be a locally convex topological vector space and $f: E \to \R \cup \{+\infty\}$ proper. Then $f = f^{**}$ if and only if $f$ is convex and lower semi-continuous.
\end{proposition}

\begin{proof}
The proof for Banach spaces is contained in the proof of Lemma~\ref{thm:LF_properties}(iv) in Section~\ref{sec:convex_analysis}. The extension to locally convex spaces requires replacing the norm topology with the given topology in the Hahn-Banach separation step: the epigraph $\mathrm{epi}(f)$ is closed and convex in $E \times \R$, and the Hahn-Banach theorem holds in locally convex spaces (see Rudin \cite{Rudin1991}, Theorem~3.4). The remainder of the argument is identical. See also Ekeland and Temam \cite{EkelandTemam1999}, Chapter~I, Proposition~4.1.
\end{proof}

\begin{proposition}[Moreau-Rockafellar]\label{thm:moreau_rockafellar}
Let $f, g: E \to \R \cup \{+\infty\}$ be convex with one continuous at a point of $\mathrm{dom}(f) \cap \mathrm{dom}(g)$. Then $\partial(f + g)(x) = \partial f(x) + \partial g(x)$.
\end{proposition}

\begin{proof}
This is Lemma~\ref{thm:subdiff_properties}(v), proved in Section~\ref{sec:convex_analysis}. See also Phelps \cite{Phelps1993}, Theorem~3.16.
\end{proof}

%
%
%

\bmhead{Acknowledgements}

The author is grateful to Stefano Luzzatto for supervision during the ICTP Postgraduate Diploma in Mathematics at the International Centre for Theoretical Physics, Trieste, Italy (2013), during which the author worked through Bowen's monograph.

%


\end{document}